\documentclass[journal,onecolumn]{IEEEtran}
\ifCLASSINFOpdf
\else
\fi

\usepackage{amssymb}
\usepackage{amsthm}

\usepackage{amsmath}

\usepackage{algorithm}
\usepackage{algorithmic}

\usepackage{bm}

\usepackage{mathrsfs}
\usepackage{diagbox}
\usepackage{mathtools}
\DeclarePairedDelimiter{\ceil}{\lceil}{\rceil}

\usepackage{color}
\def\red{\textcolor[rgb]{1,0,0}}

\newtheorem{theorem}{Theorem}
\newtheorem{lemma}{Lemma}

\newtheorem{cor}{Corollary}

\newtheorem{dfn}{Definition}

\usepackage{tikz}
\usetikzlibrary{arrows}

\newcommand{\tabincell}[2]{\begin{tabular}{@{}#1@{}}#2\end{tabular}}

\begin{document}
%
\title{Tensor Completion via a Low-Rank Approximation Pursuit}
%
%
%

\author{An-Bao~Xu  
\thanks{This work was supported by the National Natural Science Foundation of China under Grant 11801418.}
\thanks{The author is with the College of Mathematics and Physics, Wenzhou University, Zhejiang 325035, China (e-mail: xuanbao@wzu.edu.cn).}}

\maketitle

\begin{abstract}
This paper considers the completion problem for a tensor (also referred to as a multidimensional array) from limited sampling. Our greedy method is based on extending the low-rank approximation pursuit (LRAP) method for matrix completions to tensor completions. The method performs a tensor factorization using the tensor singular value decomposition (t-SVD) which extends the standard matrix SVD to tensors. The t-SVD leads to a notion of rank, called tubal-rank here. We want to recreate the data in tensors from low resolution samples as best we can here. To complete a low resolution tensor successfully  we assume that the given tensor data has low tubal-rank. For tensors of low tubal-rank, we establish convergence results for our method that are based on the tensor restricted isometry property (TRIP). Our result with the TRIP condition for tensors is similar to low-rank matrix completions under the RIP condition. The TRIP condition uses the t-SVD for low tubal-rank tensors, while RIP  uses the SVD for matrices. We show that a subgaussian measurement map satisfies the TRIP condition with high probability and gives  an almost optimal bound on the number of required measurements. We  compare the numerical performance of the proposed algorithm with those for state-of-the-art approaches on video recovery and color image recovery. \\

\end{abstract}


%
\IEEEpeerreviewmaketitle

\section{Introduction}

Tensors generalize vectors and matrices \cite{KB2009,H2012}. Tensor completions recently have drawn much attention \cite{LMWY2013,GRY2011,SDLS2013,TKK2010,PC2013,YZ2014}. 
We try to partially reconstruct the original tensor  from a given low rank tensor formed by partial observations and do so sufficiently well. Such problems arise in a variety of applications, such as in signal processing \cite{LMWY2013, ZEAHK2014}, in multi-class learning \cite{OTJ2010}, in data mining \cite{KS2008,SPLCLQ2009}, and in dimension reduction \cite{LLR1995}. 

Typically,  tensor completion is formulated as an optimization problem that involves sums of nuclear norms of the unfolding matrices inside the unknown tensor and uses the notion of tensor rank defined in terms of the higher singular value decomposition (HOSVD) \cite{LMWY2013,GRY2011,SDLS2013,TKK2010}. However,  algorithms that use  sums of matrix nuclear norms generally need to solve several subproblems via a full singular value decomposition at each iteration step. This limits their speed, especially for large tensor sizes due to the high cost of  full SVDs. Moreover sums of nuclear norms do not allow exploiting the tensor structure and such processes lead to  suboptimal procedures, see \cite{YZ2014} e. g.. Furthermore, if  the tensor rank is determined by the CANDECOM/PARAFAC decomposition (CP) for example, a best multirank-k approximation may only exist  under further assumptions. Computing the multirank-k approximation is  highly nontrivial even when it  exits, see \cite{ZSKA2018}.


 For tensor decompositions, our algebraic t-SVD  framework  differs from the classic multilinear algebraic framework \cite{H2012} and the tensor tubal rank (using t-SVD), differs from the CP rank (using CP) and the Tucker rank (using HOSVD). Therefore, bounds and conditions of tensor completion using low Tucker rank or low CP rank are not directly comparable to those in  this paper. 

The t-SVD has recently been used in \cite{LFCLLY2019,ZEAHK2014,ZA2017} for tensor recoveries applied to computer vision. These papers define tensor incoherence conditions and obtain theoretical performance bounds for the corresponding algorithms. Our method differs from these as we employ the TRIP condition to obtain  theoretical performance bounds. We were inspired by \cite{D2016}, \cite{RSS2017} and their main tools, namely $\epsilon$-nets and covering numbers. With $\epsilon$-nets and covering numbers, we show that  subgaussian measurement maps satisfy the TRIP condition with high probability and for a certain almost optimal bound on the number of measurements. Thus we can provide convergence results that hold with high probability when the TRIP conditions are satisfied.


In this paper, the Y tensor and the X tensors all have the same original large tensor sizes in three dimensions. our basic tensor completion model involves a large data tensor $\bm{\mathcal{Y}}_{n_1,n_2,n_3} $, a set $\Omega$ comprised of horizontal, lateral and frontal slice indices $\{(O_h,O_l,O_f)\}$ with  $1 \leq O_h \leq n_1$ and  $1 \leq O_l \leq n_2$ and  $1 \leq O_f \leq n_3$  and $|\Omega| \ll n_1 \cdot n_2 \cdot n_3$ that indicate  the horizontal, lateral and frontal slice indices of a entry subset of $\bm{\mathcal{Y}}$. In this framework we search  for a low rank completion tensor $\bm{\mathcal{X}}_{n_1,n_2,n_3}$ which has the same entries as $\bm{\mathcal{Y}}$ in the positions of $\Omega$ and has arbitrary entries in its complementary positions. The selection is such that  $\bm{\mathcal{X}}$ has the minimal possible tensor rank. This tensor completion problem can  be formalized as follows.  
\begin{equation}\label{1.1}
\text{Find } \ \mathop{\min}\limits_{\bm{\mathcal{X}}\in\mathbb{R}^{n_1\times n_2 \times n_3}} {\rm rank_t}(\bm{\mathcal{X}})
\quad \text{such that} \quad  P_\Omega(\bm{\mathcal{X}})=P_\Omega(\bm{\mathcal{Y}}),
\end{equation}
where $\rm rank_t$ is the tensor tubal rank of Definition \ref{definition 2.6}, see below, $\Omega$ is  the set of all index pairs  in $\bm{\mathcal{Y}}$  that $\bm{\mathcal{X}}$  shares with $\bm{\mathcal{Y}}$. Here $P_\Omega$ denotes the orthogonal projector onto the span of tensors with zeros at the positions not in $\Omega$. The aim of our tensor completion is to create a low rank tensor $\bm{\mathcal{Y}}$ from the partially observed tensor $P_\Omega(\bm{\mathcal{Y}})$ of $\bm{\mathcal{Y}}$. This tensor completion is based on that  you are given a imprecise tensor and want to recreate a same sized tensor that mimicks the nonzero part of the original tensor $\bm{\mathcal{Y}}$ as best as you can. The nonzero part of the original tensor is $P_\Omega(\bm{\mathcal{Y}})$ which is fewer than the $n_1 \cdot n_2 \cdot n_3$ measurements in $\bm{\mathcal{Y}}$. This   benefits  many applications.

Following Section 5 of \cite{WLLFDY2015}, we  extend the basic model (\ref{1.1}) to the following tensor sensing problem
\begin{equation}\label{1.2}
\text{Find }  \ \mathop{{\rm min}}\limits_{\bm{\mathcal{X}} \in \mathbb{R}^{n_1 \times n_2 \times n_3}} {\rm rank_t} (\bm{\mathcal{X}})
\quad s.t. \ \Phi(\bm{\mathcal{X}})=\Phi(\bm{\mathcal{Y}}),
\end{equation}
where $\bm{\mathcal{Y}}$ is a target low rank tensor and $\Phi  =\phi \cdot {\rm vec}$ is a linear operator. Its inverse  $\Phi^{-1} $ is the linear operator with $\Phi\Phi^{-1}(\mathbf b) = \mathbf b$ for any vector $\mathbf b$. Note that $\Phi^{-1}\Phi$ is not an identity operator. This paper proposes a simple and efficient algorithm to solve the more general problem (\ref{1.2}). In every iteration,  rank-one basis tensors are constructed by the truncated t-SVD of  the currently known residual tensor. In our standard version we fully update the weights (or the coefficients) for all rank-one tensors in the current basis set in each iteration. The most time-consuming process in this version is the truncated t-SVD computation. Therefore we recommend readers to adopt the t-SVD algorithm of \cite{LFCLLY2019}  or the rt-SVD method of \cite{ZSKA2018} here. We also adopt an economic weight updating rule from \cite{XX2017} to decrease the time and storage complexity further. Interestingly, both algorithms converge linearly. \\

The notion of the TRIP condition here and our results are new and  not directly implied by the results from matrix completions using the standard matrix RIP conditions and the matrix LRAP. \\

The main contributions of our paper are:
\begin{itemize}
  \item We propose a computationally more efficient greedy algorithm for  tensor completions, which extends the LRAP method for matrix completions to tensor completions. 
  \item This article consists in an analysis of the TRIP related to the tensor formats t-SVD for random measurement maps. We show that subgaussian measurement maps satisfy the TRIP with high probability under a certain almost optimal bound on the number of measurements. 
   \item Using this result of the TRIP condition, We show that the proposed algorithm achieve  linear convergence. 
  \item We illustrate the efficiency of the proposed algorithm for tensor completions via numerical comparison with those for state-of-the-art approaches on video recovery and color image recovery.\\
\end{itemize}

The next section introduces  notations and preliminaries. In Section III, we construct our standard algorithm and  a more economic version. Section IV defines the RIP condition and proves that subgaussian measurement ensembles satisfy the RIP condition with high probability. We prove linear convergence of our algorithms in Section V. Empirical numerical test evaluations and comparisons with other methods are presented in Section VI. They verify the efficiency of the proposed algorithms. 

\section{Notations and Preliminaries}

\subsection{Notations}


Here we denote matrices by boldface capital letters and vectors by boldface lowercase letters. Tensors are represented in Euler bold script letters. For example, a third-order tensor is represented as $\bm{\mathcal{A}}$, and its $(i,j,k)$th entry is represented as $\bm{\mathcal{A}}_{ijk}$ or $a_{ijk}$. The Matlab notation $\bm{\mathcal{A}}(i,:,:),\bm{\mathcal{A}}(:,i,:)$ and $\bm{\mathcal{A}}(:,:,i)$ are used to denote respectively the $i$-th horizontal, lateral and frontal slices. Used frequently, $A^{(i)} $ denotes compactly the frontal slice $\bm{\mathcal{A}}(:,:,i)$ and $\bm{\mathcal{A}}(i,j,:)$ denotes the tube of $i,j$ in the third tensor dimension. The Frobenius inner product of two compatible tensors is 
$$ \langle \bm{\mathcal{A}}, \bm{\mathcal{B}} \rangle=\sum\nolimits_{i_1, \cdots, i_t} {\bm{\mathcal{A}}_{i_1, \cdots, i_t}\bm{\mathcal{B}}_{i_1, \cdots, i_t}}.$$
Then $||\bm{\mathcal{A}}||_F = \sqrt{\langle \bm{\mathcal{A}}, \bm{\mathcal{A}} \rangle}$ is the corresponding Frobenius tensor norm and $d_F(\bm{\mathcal{A}}, \bm{\mathcal{B}}) = ||\bm{\mathcal{A}} - \bm{\mathcal{B}}||_F$ the induced tensor metric. For any $\bm{\mathcal{A}} \in \mathbb{C}^{n_1 \times n_2 \times n_3}$, the complex conjugate of $\bm{\mathcal{A}}$ is denoted as ${\rm conj}(\bm{\mathcal{A}})$, whose entries are the complex conjugates of the respective entry in $\bm{\mathcal{A}}$. The vector ${\rm vec}(\bm{\mathcal{Y}})$ contains the entries  of $\bm{\mathcal{Y}}$ reshaped  by concatenating all tensor entries. The vector $\mathop {\textbf{y}}\limits^{\textbf{.}}={\rm vec}_\Omega(\bm{\mathcal{Y}})=\{(y_{w_1}, \cdots, y_{w_{|\Omega|}})^T \ \forall w_i \in\Omega\}$ denotes the vector generated by concatenating all elements of $\bm{\mathcal{Y}}$ in the index set $\Omega$.

\subsection{Discrete Fourier Transformation (DFT)}


The DFT on $\mathbf v \in \mathbb{R}^n$, denoted as $\hat{\textbf v} $, is obtained by
$$\hat{\mathbf v} = \mathbf F_n \mathbf v \in \mathbb{C}^n,$$ 
where $\mathbf F_n$ is the DFT matrix 
$$\mathbf F_n = \begin{bmatrix}
     1 &   1 &  1& \cdots & 1 \\
     1 &  \omega & \omega^2 & \cdots  & \omega^{n-1} \\
   \vdots  & \vdots & \vdots &   \ddots & \vdots \\
      1 &  \omega^{n-1}  & \omega^{2(n-1)}  & \cdots  & \omega^{(n-1)(n-1)} \\
\end{bmatrix}  \in \mathbb{C}^{n \times n},$$
and $\omega = e^{-{2 \pi i \over n}}$ is a primitive $n$-th root of unity with $i = \sqrt{-1}$ \cite{LFCLLY2019}.
Note that $\mathbf F_n / \sqrt{n}$ is an orthogonal matrix, i.e.,
\begin{equation}\label{7}
\mathbf F^*_n \mathbf F_n = \mathbf F_n \mathbf F^*_n = n \mathbf I_n.
\end{equation}
Hence $\mathbf F^{-1}_n = \mathbf F^*_n /n$. 
 

 

For any tensor $\bm{\mathcal{A}} \in \mathbb{R}^{n_1 \times n_2 \times n_3}$,  ${\bm{\hat\mathcal{A}}} \in \mathbb{C}^{n_1 \times n_2 \times n_3}$  denotes the result of the DFT on $\bm{\mathcal{A}}$ along the 3-rd dimension, i.e., using the DFT on all  $i,j$ tubes in $\bm{\mathcal{A}}$. By using the Matlab command ${\rm fft}$, we obtain
$$\hat{\bm{\mathcal{A}}} = {\rm fft}(\bm{\mathcal{A}},[],3).$$
Analogously $\bm{\mathcal{A}}$ is computed from $\hat{\bm{\mathcal{A}}}$ by performing the inverse FFT, i.e. ,
$$\bm{\mathcal{A}} = {\rm ifft}(\hat{\bm{\mathcal{A}}},[],3).$$
 $\overline{\mathbf A} \in \mathbb{C}^{n_1 n_3 \times n_2 n_3}$ is the block diagonal matrix whose $i$-th diagonal block is the $i$-th frontal slice $\hat{\mathbf A}^{(i)} \text{ of } \hat{\bm{\mathcal{A}}}$, i.e., 
 $$\overline{\mathbf A} = {\rm bdiag}(\hat{\bm{\mathcal{A}}}) = \begin{bmatrix}
     \hat{\mathbf A}^{(1)} &   0 & \cdots & 0 \\
     0 &  \hat{\mathbf A}^{(2)}  &  0 & \cdots \\
   \vdots  &  \ddots &   \ddots & \ddots \\
      0 &   \cdots   & 0  & \hat{\mathbf A}^{(n)} \\
\end{bmatrix}.$$
Here ${\rm bdiag}$ denotes the operator that maps the tensor $\hat{\bm{\mathcal{A}}}$ to the block diagonal matrix $\overline{\mathbf A}$. The block circulant matrix ${\rm bcirc}(\bm{\mathcal{A}}) \in \mathbb{R}^{n_1 n_3 \times n_2 n_3}$ of $\bm{\mathcal{A}}$ is given by
$${\rm bcirc}(\bm{\mathcal{A}})  = \begin{bmatrix}
     \mathbf A^{(1)} &   \mathbf A^{(n_3 )} & \cdots & \mathbf A^{(2)} \\
     \mathbf A^{(2)} & \mathbf A^{(1)}  &  \cdots & \mathbf A^{(3)} \\
   \vdots  &  \vdots &   \ddots & \vdots \\
      \mathbf A^{(n_3 )} &   \mathbf A^{(n_3 -1)}   &  \cdots  & \mathbf A^{(1)} \\
\end{bmatrix}.$$
 The block circulant matrix can be block diagonalized, i.e. ,
 \begin{equation}\label{10}
 (\mathbf F_{n_3} \otimes \mathbf I_{n_1}) \cdot {\rm bcirc}(\bm{\mathcal{A}}) \cdot (\mathbf F^{-1}_{n_3}  \otimes \mathbf I_{n_2}) = \overline{\mathbf A} 
 \end{equation}
 where $\otimes$ denotes the Kronecker product and $(\mathbf F_{n_3} \otimes \mathbf I_{n_1}) / \sqrt{n_3}$ is orthogonal \cite{LFCLLY2019}. 
Based on (\ref{7}), we have 

\begin{equation}\label{12}
||\bm{\mathcal{A}}||_F = {1 \over \sqrt{n_3}} ||\overline{\mathbf A}||_F ,
\end{equation}
 
 \begin{equation}\label{13}
 \langle \bm{\mathcal{A}}, \bm{\mathcal{B}} \rangle = {1 \over n_3}  \langle \overline{\mathbf A}, \overline{\mathbf B} \rangle \ .
 \end{equation}
 
 \subsection{T-product and T-SVD}
 
 For $\bm{\mathcal{A}} \in \mathbb{R}^{n_1 \times n_2 \times n_3}$, the ${\rm unfold}$ operator map of $\bm{\mathcal{A}}$ is the matrix of size $n_1 n_3 \times n_2$
  $${\rm unfold}(\bm{\mathcal{A}}) = \begin{bmatrix} \mathbf A^{(1)} \\  \mathbf A^{(2)}  \\ \vdots \\ \mathbf A^{(n_3)} \end{bmatrix}.$$ 
  Its inverse operator ${\rm fold}$ operates so that  $${\rm fold}({\rm unfold}(\bm{\mathcal{A}})) = \bm{\mathcal{A}}.$$
 \begin{dfn}\label{definition 2.1} \cite{KM2011}
 \textbf{ (T-product)}
 Let $\bm{\mathcal{A}} \in \mathbb{R}^{n_1 \times n_2 \times n_3}$ and $\bm{\mathcal{B}} \in \mathbb{R}^{n_2 \times l \times n_3}$. Then the t-product $\bm{\mathcal{A}}  * \bm{\mathcal{B}} $ is the following tensor of size $n_1 \times l \times n_3$ :
 \begin{equation}\label{14}
 \bm{\mathcal{A}} * \bm{\mathcal{B}} = {\rm fold}({\rm bcirc}(\bm{\mathcal{A}}) \cdot {\rm unfold}(\bm{\mathcal{B}}))\ .
 \end{equation}
\end{dfn}

The t-product is analogous to  matrix multiplication. The only difference between them is that the circular convolution substitutes multiplication between  elements. Therefore the t-product of tensors is  matrix multiplication in the Fourier domain; namely, $\bm{\mathcal{C}}  = \bm{\mathcal{A}}  * \bm{\mathcal{B}} $ is equivalent to $\overline{\mathbf C} = \overline{\mathbf A} \cdot \overline{\mathbf B}$ due to (\ref{10}) \cite{LFCLLY2019}. 
This property provides an efficient way (based on the FFT) to calculate t-product instead of performing (\ref{14}). See Algorithm 1 in \cite{LFCLLY2019}. The t-product has many similar properties as the matrix-matrix product.
  
 \begin{dfn}\label{definition 2.2}\cite{LFCLLY2019}
\textbf{ (Conjugate transpose)} The conjugate transpose of a tensor $\bm{\mathcal{A}} \in \mathbb{C}^{n_1 \times n_2 \times n_3}$ is the tensor $\bm{\mathcal{A}}^* \in \mathbb{C}^{n_1 \times n_2 \times n_3}$ obtained by conjugate transposing each of the frontal slices and then reversing the order of transposed frontal slices 2 through $n_3$.
 \end{dfn}
 
\begin{dfn}\label{definition 2.3}\cite{KM2011}
\textbf{ (Identity tensor)} The identity tensor $\bm{\mathcal{I}} \in \mathbb{R}^{n \times n \times n_3}$ is the tensor with  the $n \times n$ identity matrix $\mathbf I_n$ as its first frontal slice  and all other frontal slices being $\mathbf O_n$. 
\end{dfn}
 
 It is obviously that $\bm{\mathcal{A}}  * \bm{\mathcal{I}} = \bm{\mathcal{A}}$ and $\bm{\mathcal{I}}  * \bm{\mathcal{A}} = \bm{\mathcal{A}}$ for appropriate dimensions. For the tensor $\hat{\bm{\mathcal{I}}} = {\rm fft}(\bm{\mathcal{I}}, [], 3)$ for example, the frontal slice is the identity matrix.
 
\begin{dfn}\label{definition 2.4}\cite{KM2011}
\textbf{ (Orthogonal tensor)} A tensor $\bm{\mathcal{Q}} \in \mathbb{R}^{n \times n \times n_3}$ is orthogonal if it satisfies $\bm{\mathcal{Q}}^* * \bm{\mathcal{Q}} = \bm{\mathcal{Q}} * \bm{\mathcal{Q}}^* = \bm{\mathcal{I}}.$
\end{dfn}
 
 The notion of partial orthogonality can be defined and is analogous to that a tall, thin matrix has orthogonal columns \cite{KM2011}. In this case if $\bm{\mathcal{Q}}$ is $n \times q \times n_3$ and \textbf{partially orthogonal}, then this implies that $\bm{\mathcal{Q}}^* * \bm{\mathcal{Q}}$ is well defined and equivalent to the $q \times q \times n_3$ identity tensor. 
 
 \begin{dfn}\label{definition 2.5}\cite{KM2011}
\textbf{ (F-diagonal tensor)} A tensor is called f-diagonal if each of its frontal slices is a diagonal matrix.
\end{dfn}
 
 \begin{theorem}\label{Theorem 2.2}\cite{LFCLLY2019}
\textbf{ (T-SVD)} Let $\bm{\mathcal{A}} \in \mathbb{R}^{n_1 \times n_2 \times n_3}$. Then it can be factored as  
 \begin{equation}\label{16}
 \bm{\mathcal{A}} = \bm{\mathcal{U}} * \bm{\mathcal{S}} * \bm{\mathcal{V}}^*
  \end{equation}
 where $\bm{\mathcal{U}} \in \mathbb{R}^{n_1 \times n_2 \times n_3}$ and $\bm{\mathcal{V}} \in \mathbb{R}^{n_1 \times n_2 \times n_3}$ are orthogonal, and $\bm{\mathcal{S}} \in \mathbb{R}^{n_1 \times n_2 \times n_3}$ is an f-diagonal tensor.
  \end{theorem}

 Theorem \ref{Theorem 2.2} implies that any 3 order tensor can be factorized into 3 factors, two of which are orthogonal tensors and the central factor is an f-diagonal tensor, as depicted in Figure \ref{Figure 1}.


 
 \begin{figure}[h]
\centering
 \includegraphics[scale=0.30]{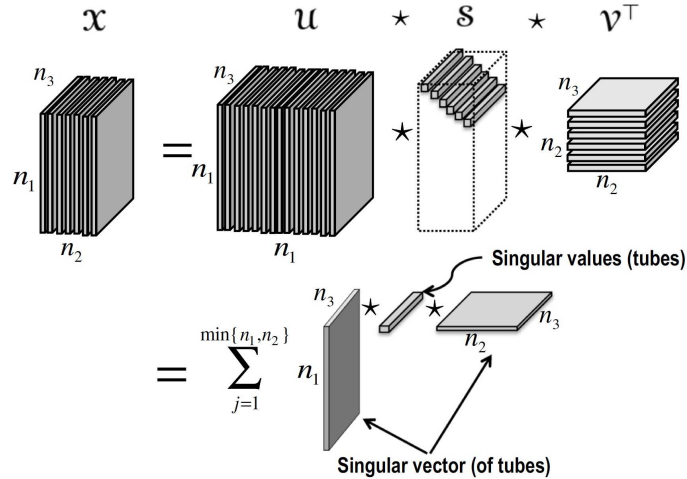}  
  \caption{The construction of the t-SVD is similar to the matrix substitute the equivalent matrix operations. Like the matrix SVD, the t-SVD can also be formalized as the sum of outer tensor products}
\label{Figure 1}
\end{figure}

  \begin{dfn}\label{definition 2.6} \cite{LFCLLY2019}
 	\textbf{ (Tensor tubal rank)} For $\bm{\mathcal{A}} \in \mathbb{R}^{n_1 \times n_2 \times n_3}$ the tensor tubal rank, denoted as ${\rm rank_t} (\bm{\mathcal{A}})$, is the number of nonzero singular tubes of $\bm{\mathcal{S}}$, where $\bm{\mathcal{S}}$ is the central factor of the  t-SVD of $	\bm{\mathcal{A}} = \bm{\mathcal{U}} * \bm{\mathcal{S}} * \bm{\mathcal{V}}^* $.  
 \end{dfn}
 
 \begin{dfn}\label{definition 2.9}\cite{ZSKA2018}
  \textbf{(Truncated t-SVD)} Given a tensor $\bm{\mathcal{A}} \in \mathbb{R}^{n_1 \times n_2 \times n_3}$, the truncated t-SVD of $\bm{\mathcal{A}}$ is $$\bm{\mathcal{A}}_k = \sum^k_{i=1} \bm{\mathcal{U}}(:,i,:) * \bm{\mathcal{S}}(i,i,:) * \bm{\mathcal{V}}(:,i,:)^T,$$ where $k \leq {\rm min}(n_1, n_2)$ is a target truncation term. Here $\bm{\mathcal{A}}_k = \bm{\mathcal{U}}_k * \bm{\mathcal{S}}_k * \bm{\mathcal{V}}_k^T,$
   where $\bm{\mathcal{U}}_k \in \mathbb{R}^{n_1 \times k \times n_3}$ and $\bm{\mathcal{V}}_k \in \mathbb{R}^{n_2 \times k \times n_3}$ are  partially orthogonal tensors and $\bm{\mathcal{S}}_k \in \mathbb{R}^{s \times k \times n_3}$ is an f-diagonal tensor. 
 \end{dfn}
 
 A nice feature of the truncated t-SVD is that it  gives us the best "multirank-k" approximation of  tensors \cite{ZSKA2018}. If a tensor  $\bm{\mathcal{A}}$ has tubal rank $r$, we have $\bm{\mathcal{A}}$ = $\bm{\mathcal{A}}_r$ and we can use $\bm{\mathcal{A}}_r$  as  a reduced version of the t-SVD \cite{ZA2017} . 
 
 \begin{lemma}\label{Lemma 2.2}
 	Let $\bm{\mathcal{M}}_{1}=  \bm{\mathcal{U}} (:,1,:) * {\bm{\mathcal{S}} (1,1,:) \over ||\bm{\mathcal{S}} (1,1,:)||_F}* \bm{\mathcal{V}} (:,1,:)^T$ be the rank-one tensor from the t-SVD of $\bm{\mathcal{A}}$ .\\ Then
 	$\langle {\bm{\mathcal{M}}_1, \bm{\mathcal{A}}} \rangle  \ge {\|  \bm{\mathcal{A}} \|_F \over \sqrt{{\rm min}(m,n)}}$ for all $k \ge 1$.
 \end{lemma}
 \proof
 The optimum $\bm{\mathcal{M}}_1$ in our algorithm satisfies
 $$\langle \bm{\mathcal{M}}_1, \bm{\mathcal{A}}\rangle^2 =  ||\bm{\mathcal{S}} (1,1,:)||^2_F \cdot \langle \bm{\mathcal{M}}_1, \bm{\mathcal{M}}_1 \rangle^2 = ||\bm{\mathcal{S}} (1,1,:)||^2_F$$
 $$= || ({\mathbf F_{n_3} \over \sqrt{n_3}} \otimes I) \ {\rm unfold}(\bm{\mathcal{S}} (1,1,:)) ||^2_F$$
 $$= {1 \over n_3} ||\hat{\mathbf S}^{(1)}(1,1)||^2_F+ \cdots + {1 \over n_3} ||\hat{\mathbf S}^{(n_3)}(1,1)||^2_F$$
 $$\geq  {1 \over n_3} {\sum_i [\hat{\mathbf S}^{(1)}(i,i)]^2 \over {\rm rank}(\hat{\mathbf S}^{(1)} ) }+ \cdots +  {1 \over n_3} {\sum_i [\hat{\mathbf S}^{(n_3)}(i,i)]^2 \over {\rm rank}(\hat{\mathbf S}^{(n_3)} ) }$$
 $$\geq  {1 \over n_3} {\sum_i [\hat{\mathbf S}^{(1)}(i,i)]^2 \over {\rm min}(n_1,n_2) }+ \cdots +  {1 \over n_3} {\sum_i [\hat{\mathbf S}^{(n_3)}(i,i)]^2 \over {\rm min}(n_1,n_2) }$$
 $$=  {1 \over n_3} { ||\overline{\mathbf S}||^2_F \over {\rm min}(n_1,n_2) }=  {1 \over n_3} { ||\overline{\mathbf A}||^2_F \over {\rm min}(n_1,n_2) } =  {1 \over n_3} { n_3 ||\bm{\mathcal{A}}||^2_F \over {\rm min}(n_1,n_2) }.$$ This completes the proof. $\Box$
 
 This result will be used to prove Theorem \ref{Theorem 3.1}.
 
\section{Low-Rank Approximation Pursuit  for Tensor Sensing (LRAP4TS)}

Based on Theorem \ref{Theorem 2.2} and Figure \ref{Figure 1}, any tensor $\bm{\mathcal{X}}\in \mathbb{R}^{n_1\times n_2 \times n_3}$  can be written as a linear combination of rank-one tensors, namely
$$ \bm{\mathcal{X}}= \bm{\mathcal{U}} * \bm{\mathcal{S}} * \bm{\mathcal{V}}^T=\mathop{\sum}\limits_{i=I} \bm{\mathcal{U}}(:,i,:) * \bm{\mathcal{S}}(i,i,:) * \bm{\mathcal{V}}(:,i,:)^T$$
$$=\mathop{\sum}\limits_{i=I} ||\bm{\mathcal{S}}(i,i,:)||_F \cdot \bm{\mathcal{U}}(:,i,:) * {\bm{\mathcal{S}}(i,i,:) \over ||\bm{\mathcal{S}}(i,i,:)||_F}* \bm{\mathcal{V}}(:,i,:)^T$$
$$=\mathop{\sum}\limits_{i=I}\theta_i\bm{\mathcal{M}}_i=\bm{\mathcal{M}} (\bm{\theta}) $$
where $\{\bm{\mathcal{M}}_i = {\bm{\mathcal{S}}(i,i,:) \over ||\bm{\mathcal{S}}(i,i,:)||_F} : i\in I \}$ is the set of all $n_1\times n_2\times n_3$ rank-one tensors with unit Frobenius norm and $\theta_i = ||\bm{\mathcal{S}}(i,i,:)||_F $ is the norm of the $i$-th singular value tube.  
Hence the original low rank tensor sensing problem (\ref{1.2}) can be rewritten as
\begin{equation}\label{2.2}
\mathop{\min}\limits_{\bm{\theta}} \| {\bm{\theta}} \|_0
\quad \text{s.t.} \quad \Phi(\bm{\mathcal{M}}(\bm{\theta}))=\Phi(\bm{\mathcal{Y}}),
\end{equation}
where $\| \bm{\theta}\|_0$ denotes the number of nonzero elements of vector $ \bm{\theta}$.
An alternative formula of Problem (\ref{2.2}) is
\begin{equation*}
\mathop{\min}\limits_{\bm{\theta}} \| \Phi(\bm{\mathcal{M}}(\bm{\theta}))-\Phi(\bm{\mathcal{Y}}) \|_F^2 \quad \text{s.t.} \quad \ \| {\bm{\theta}} \|_0 \leqslant r.
\end{equation*}
This problem can be solved by our algorithm, which is an LRAP \cite{XX2017} type algorithm using rank-one tensors as the basis. Below we show the main steps of our LRAP4TS and its economic version ELRAP4TS in Algorithm \ref{alg:Framwork}. Let the orthogonal projector $P_{\Omega}$ be the linear operator $\Phi$, then Algorithm \ref{alg:Framwork} is suitable  for tensor completion and we will be refer to it as LRAP4TC or ELRAP4TC.

Now we give details of the iteration procedure details for both versions of the tensor sensing algorithm. Both greedy algorithms add  $s$ new rank-one tensors to the basis set in each iteration. Both algorithms alternate between three iteration steps: (1) pursuit s rank-one basis tensors;  (2) update the weights of the tensors; and (3) renew the residual tensor.

\begin{algorithm*}[htb]
\caption{Low-Rank Approximation Pursuit  for Tensor Sensing (LRAP4TS) and Economic Version (ELRAP4TS)}
\label{alg:Framwork}
\begin{algorithmic}[1]
\REQUIRE ~~ $\bm{\mathcal{R}}_0 = \Phi^{-1} \Phi (\bm{\mathcal{Y}})$, the tubal rank $r$ of the estimated tensor $\bm{\mathcal{Y}}$ and the number $s$ of candidates searched in each iteration.\\
\textbf{Initialize:} Set $\bm{\mathcal{X}}_0=0$, $\bm{\mathcal{R}}_1=\bm{\mathcal{R}}_0$, $\bm{\theta}^0=0$, $\widehat{\bm{\mathcal{Y}}}_{0}=0$ and $k=1$.\\
\underline{{\bf{While}}} \  \  $k \leqslant \ceil{r/s}$ \ \textbf{do :} \\[0mm] 

\qquad \textbf{Step 1:}  \begin{minipage}[t]{150mm} {(This step  yields  the best multirank-s approximation of $\bm{\mathcal{R}}_k$) \\[-3mm] 

 Search for the $s$ leading principle left and right singular vectors of tubes $\{ (\bm{\mathcal{U}}_k (:,j,:)$, $\bm{\mathcal{S}}_k (j,j,:), \bm{\mathcal{V}}_k (:,j,:)), j=1,\cdots, s \}$ 
 of $\bm{\mathcal{R}}_k$, and set up the $s$ rank-one basis tensors $\{ \bm{\mathcal{M}}_{k,j}=\bm{\mathcal{U}}_k (:,j,:) * {\bm{\mathcal{S}}_k (j,j,:) \over ||\bm{\mathcal{S}}_k (j,j,:)||_F}* \bm{\mathcal{V}}_k (:,j,:)^T, j=1,\cdots, s \}$.\\[0mm]
} \end{minipage} \\[-2mm]
\qquad \textbf{Step 2:} \begin{minipage}[t]{150mm} {Solve the following least squares problem:\\ [-4mm]
\begin{equation}\label{2.6}
1) \mathop{\min}\limits_{\bm{\theta}=(\theta_{1,1},\cdots,\theta_{1,s}, \cdots, \theta_{k,1},\cdots \theta_{k,s})^T \in \mathbb{R}^{sk}}\| \sum_{i=1}^{k} {[\theta_{i,1} \Phi^{-1} \Phi (\bm{\mathcal{M}}_{i,1}) + \cdots + \theta_{i,s}\Phi^{-1} \Phi (\bm{\mathcal{M}}_{i,s})] - \bm{\mathcal{R}}_0} \|^2. \ \ ({\rm LRAP4TS}) 
\end{equation}\\ [-4mm]
Or\\ [-4mm]
\begin{equation}\label{2.9}
2) \mathop{\min}\limits_{\bm{\alpha}=(\alpha_0, \alpha_1,\cdots, \alpha_s )^T \in \mathbb{R}^{s}}\| \alpha_0 \bm{\mathcal{X}}_{k-1} + \alpha_1 \Phi^{-1} \Phi (\bm{\mathcal{M}}_{k,1} +\cdots+\alpha_s \Phi^{-1} \Phi (\bm{\mathcal{M}}_{k,s} )-\bm{\mathcal{R}}_0 \|^2. \ \ ({\rm ELRAP4TS}) 
\end{equation} \\
} \end{minipage} \\[-4mm]
\qquad \textbf{Step 3:} \begin{minipage}[t]{150mm} {1) Set $\bm{\mathcal{X}}_{k}=\sum_{i=1}^{k} [\theta_{i,1}^{k} \Phi^{-1} \Phi(\bm{\mathcal{M}}_{i,1})+ \cdots + \theta_{i,s}^{k} \Phi^{-1} \Phi(\bm{\mathcal{M}}_{i,s})]$ and $\bm{\mathcal{R}}_{k+1}= \Phi^{-1} \Phi(\bm{\mathcal{Y}}) - \bm{\mathcal{X}}_{k}$; $k \leftarrow k+1$\\[2mm] 
		Or\\[2mm]
		2) Set $\bm{\mathcal{X}}_{k}=\alpha_0^k \bm{\mathcal{X}}_{k-1} + \alpha_1^k \Phi^{-1} \Phi(\bm{\mathcal{M}}_{k,1} )+\cdots+\alpha_s^k \Phi^{-1} \Phi(\bm{\mathcal{M}}_{k,s} )$, \\[2mm]
		\hspace*{5mm} $\widehat{\bm{\mathcal{Y}}}_{k}=\alpha_0^k \widehat{\bm{\mathcal{Y}}}_{k-1} + \alpha_1^k \bm{\mathcal{M}}_{k,1} +\cdots+\alpha_s^k \bm{\mathcal{M}}_{k,s}$ and $\bm{\mathcal{R}}_{k+1}= \Phi^{-1} \Phi(\bm{\mathcal{Y}}) - \bm{\mathcal{X}}_{k}$; $k \leftarrow k+1$.\\}\end{minipage}  \\[-4mm]

\underline{\textbf{End While}}\\[0mm] 
\ENSURE ~~ For 1) LRAP4TS:  Construct tensor $\widehat{\bm{\mathcal{Y}}}=\sum_{i=1}^{k}[\theta_{i,1}^{k} (\bm{\mathcal{M}}_{i,1})+ \cdots + \theta_{i,s}^{k} (\bm{\mathcal{M}}_{i,s})]$.\\[0mm] 
\hspace*{10mm} For 2) ELRAP4TS:  $\widehat{\bm{\mathcal{Y}}}=\widehat{\bm{\mathcal{Y}}}_{k}$.\\[-2mm]
\end{algorithmic}
\end{algorithm*}

 (1) In this step, we find a set of $s$  rank-one basis tensors \{$\bm{\mathcal{M}}_{k,1},\cdots, \bm{\mathcal{M}}_{k,s}$\} with unit Frobenius norm, which are related to the currently known residual tensor $\bm{\mathcal{R}}_k$. The tensors $\{\bm{\mathcal{M}}_{k,1},\cdots, \bm{\mathcal{M}}_{k,s}\}$  are constructed based on the $s$ leading principle left and right  singular vectors of the tubes $\{ (\bm{\mathcal{U}}_k (:,j,:), \bm{\mathcal{S}}_k (j,j,:), \bm{\mathcal{V}}_k (:,j,:)), j=1,\cdots, s \}$ from  the t-SVD of $\bm{\mathcal{R}}_k$. By construction, the $s$ rank-one basis tensors $\{ \bm{\mathcal{M}}_{k,j}=\bm{\mathcal{U}}_k (:,i,:) * {\bm{\mathcal{S}}_k (i,i,:) \over ||\bm{\mathcal{S}}_k (i,i,:)||_F}* \bm{\mathcal{V}}_k (:,i,:)^T, j=1,\cdots, s \}$  are orthogonal to each other and have Frobenius norm one. This step locates  the best multirank-s approximation to $\bm{\mathcal{R}}_k$  by computing a truncating t-SVD. The t-SVD can be obtained by the  t-SVD algorithm or the rt-SVD method.

(2) In the standard version of this step, the weights $\bm{\theta}^k$ are estimated for all current basis tensors $\{\bm{\mathcal{M}}_{i,1},\cdots, \bm{\mathcal{M}}_{i,s}, i=1,\cdots, k\}$ by solving the least squares problem (\ref{2.6}).

 Using the orthogonal projector $P_{\Omega}$ as the linear operator $\Phi$ in formula (\ref{1.2}), Problem (\ref{2.6}) can be rewritten as 
\begin{equation*}
\mathop{\min}\limits_{\bm{\theta} \in \mathbb{R}^{sk}} \| \sum_{i=1}^{k} \sum_{j=1}^{s} {\theta_{i,j} P_{\Omega}(\bm{\mathcal{M}}_{i,j}) - P_{\Omega}(\bm{\mathcal{Y}})} \|^2.
\end{equation*}       
By reshaping the tensors $P_{\Omega}(\bm{\mathcal{Y}})$ and $P_{\Omega}(\bm{\mathcal{M}}_{i,j})$ into column vector form $\mathop {\textbf{y}}\limits^{\textbf{.}}$ and ${\mathop {\textbf{m}}\limits^{\textbf{.}}}_{i,j}$,  the above overdetermined system can be reformulated as   
$\mathop{\min}\limits_{\bm{\theta} \in \mathbb{R}^{sk}} \| \overline {\mathbf M}_k \bm{\theta}  - \mathop {\textbf{y}}\limits^{\textbf{.}} \|^2$.
Here $\overline {\mathbf M}_k = [ {\mathop {\textbf{m}}\limits^{\textbf{.}}}_{1,1},\cdots,{\mathop {\textbf{m}}\limits^{\textbf{.}}}_{1,s}, \cdots, {\mathop {\textbf{m}}\limits^{\textbf{.}}}_{k,1},\cdots, {\mathop {\textbf{m}}\limits^{\textbf{.}}}_{k,s} ]$ is the matrix formed by all reshaped basis tensors. The row size of $\overline {\mathbf M}_k$ is equal to the total number of observed entries $p=|\Omega|$.

In case of using  ELRAP4TS,  the orthogonal projection step consists of only tracking the estimated tensor $\bm{\mathcal{X}}_{k-1}$ and the $s$ rank-one basis tensors $\bm{\mathcal{M}}_{k,j}$ for $j=1,\cdots,s$. This step of ELRAP4TS updates the weights for  $s+1$ tensors based on the solution of the least squares problem (\ref{2.9}).

(3) Here we  update the residual tensor $\bm{\mathcal{R}}_{k+1}=\mathcal{P}_\Omega(\bm{\mathcal{Y}})-\bm{\mathcal{X}}_{k}$ as follows:  In the standard LRAP4TS algorithm we use 
\begin{equation*}
\bm{\mathcal{X}}_{k}=\Phi(\bm{\mathcal{M}}(\bm{\theta}^{k}))=\sum_{i=1}^{k}[\theta_{i,1}^{k} \Phi(\bm{\mathcal{M}}_{i,1})+ \cdots + \theta_{i,s}^{k} \Phi(\bm{\mathcal{M}}_{i,s})]=
\end{equation*}
$$=\sum_{i=1}^{k}\sum_{j=1}^{s}\theta_{i,j}^{k} \Phi(\bm{\mathcal{M}}_{i,j}).$$
In the ELRAP4TS version we use  
\begin{equation*}
\bm{\mathcal{X}}_{k}=\alpha_0^k X_{k-1} + \alpha_1^k \Phi(\bm{\mathcal{M}}_{k,1} )+\cdots+\alpha_s^k \Phi(\bm{\mathcal{M}}_{k,s} ).
\end{equation*}

These  three steps are performed iteratively until our stopping criterion is reached. The Flow charts  in Figures \ref{Figure 2} and  \ref{Figure 3} illustrate the process for  $r=4$ and $s=2$.  Both of our methods  find a rank-4 tensor as the approximate solution. For Problem (\ref{1.2}) we need just 2 iterations in LRAP4TS  and in ELRAP4TS. Figure \ref{Figure 2} shows the process of LRAP4TS: the red arrows represent the first iteration step and the blue arrows illustrate the second step. After two iterations LRAP4TS has computed 4 basis tensors and 4 coefficients (or weights), as well as  the rank-4 tensor $\widehat{\bm{\mathcal{Y}}}=\theta_{1,1}^2 \bm{\mathcal{M}}_1 + \theta_{1,2}^2 \bm{\mathcal{M}}_2 + \theta_{1,3}^2 \bm{\mathcal{M}}_3 +\theta_{1,4}^2 \bm{\mathcal{M}}_4 $ (when $r=4$). If we were to continue iterating then in the $h$-th iterations LRAP4TS algorithm would have  to handle 2$h$ basis tensors and $2h$ weights. This data growth induces expanded storage requirements for these tensors  and it increases the computational complexity. Figure \ref{Figure 3} illustrates that ELRAP4TS builds only 3 basis tensors and computes only 3 coefficients in each iteration which is a slight improvement over  LRAP4TS. 

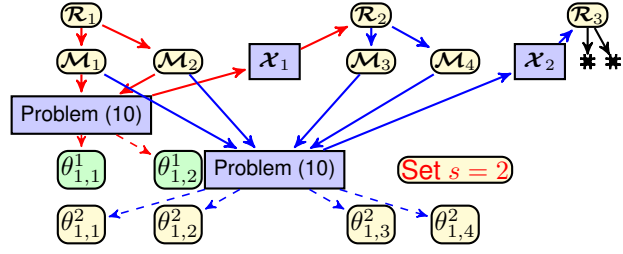
\begin{figure}

\begin{center}
\begin{tikzpicture}[
  font=\sffamily,
  every matrix/.style={ampersand replacement=\&,column sep=0.01cm,row sep=0.2cm},
  source/.style={draw,thick,rounded corners,fill=yellow!20,inner sep=.04cm},
  process/.style={draw,thick, rectangle,fill=blue!20},
  sink/.style={source,fill=green!20},
  datastore/.style={draw,very thick,shape=datastore,inner sep=.04cm},
  dots/.style={gray,scale=2},
  to/.style={->,>=stealth',shorten >=1pt,semithick,font=\sffamily\footnotesize},
  every node/.style={align=center}]

  \matrix{
          \node[source] (R1) {\footnotesize$\bm{\mathcal{R}}_1$};
    \&
    \&
    \&   \node[source] (R2) {\footnotesize$\bm{\mathcal{R}}_2$}; 
    \&
    \&
    \&   \node[source] (R3) {\footnotesize$\bm{\mathcal{R}}_3$}; \\

          \node[source] (M1) {\footnotesize$\bm{\mathcal{M}}_1$};
    \&   \node[source] (M2) {\footnotesize$\bm{\mathcal{M}}_2$};
    \&   \node[process] (X1){\footnotesize$\bm{\mathcal{X}}_1$}; 
    \&   \node[source] (M3) {\footnotesize$\bm{\mathcal{M}}_3$};
    \&   \node[source] (M4) {\footnotesize$\bm{\mathcal{M}}_4$};
    \&   \node[process] (X2){\footnotesize$\bm{\mathcal{X}}_2$};
    \&   \node[source] (Space1) {};
    \&   \node[source] (Space2) {};  \\

          \node[process] (LS1){\footnotesize Problem (\ref{2.6})};  \\

          \node[sink] (theta111) {$\theta_{1,1}^1$};
     \&  \node[sink] (theta121) {$\theta_{1,2}^1$};
     \&  \node[process] (LS2){\footnotesize Problem (\ref{2.6})};
     \&
     \&   \node[source] (s){\red{Set $s=2$}};\\

          \node[source] (theta112) {$\theta_{1,1}^2$};
    \&   \node[source] (theta122) {$\theta_{1,2}^2$};
    \&   
    \&   \node[source] (theta132) {$\theta_{1,3}^2$};
    \&   \node[source] (theta142) {$\theta_{1,4}^2$};

     \\ };

 \draw[to,red,thick] (R1) --(M1);        
 \draw[to,red,thick] (R1) --(M2);           
 \draw[to,red,thick] (M1) --(LS1);        
 \draw[to,red,thick] (M2) --(LS1); 
 \draw[to,red,dashed] (LS1) --(theta111);        
 \draw[to,red,dashed] (LS1) --(theta121);      
 \draw[to,red,thick] (LS1) --(X1);  
 \draw[to,red,thick] (X1) --(R2);

 \draw[to,blue,thick] (R2) --(M3);        
 \draw[to,blue,thick] (R2) --(M4);       
 \draw[to,blue,thick] (M1) --(LS2);        
 \draw[to,blue,thick] (M2) --(LS2);     
 \draw[to,blue,thick] (M3) --(LS2);        
 \draw[to,blue,thick] (M4) --(LS2); 
 \draw[to,blue,dashed] (LS2) --(theta112);        
 \draw[to,blue,dashed] (LS2) --(theta122);     
 \draw[to,blue,dashed] (LS2) --(theta132);        
 \draw[to,blue,dashed] (LS2) --(theta142);      
 \draw[to,blue,thick] (LS2) --(X2);  
 \draw[to,blue,thick] (X2) --(R3);

 \draw[to,thick] (R3) --(Space1);        
 \draw[to,thick] (R3) --(Space2);

\end{tikzpicture}

\end{center}
\caption{Standard LRAP4TS Algorithm (With $r=4$ and $s=2$)}\label{Figure 2}
\end{figure}


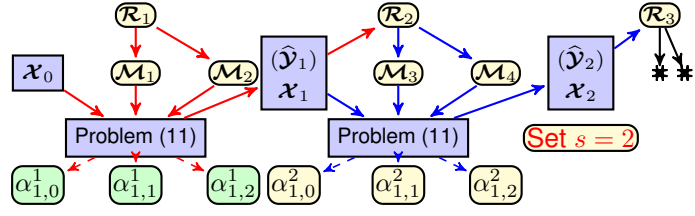
\begin{figure}

\begin{center}
\begin{tikzpicture}[
  font=\sffamily,
  every matrix/.style={ampersand replacement=\&,column sep=0.001cm,row sep=0.1cm},
  source/.style={draw,thick,rounded corners,fill=yellow!20,inner sep=.04cm},
  process/.style={draw,thick, rectangle,fill=blue!20},
  sink/.style={source,fill=green!20},
  datastore/.style={draw,very thick,shape=datastore,inner sep=.04cm},
  dots/.style={gray,scale=2},
  to/.style={->,>=stealth',shorten >=1pt,semithick,font=\sffamily\footnotesize},
  every node/.style={align=center}]

  \matrix{

    \&   \node[source] (R1) {\footnotesize$\bm{\mathcal{R}}_1$};
    \&
    \&
    \&   \node[source] (R2) {\footnotesize$\bm{\mathcal{R}}_2$}; 
    \&
    \&
    \&   \node[source] (R3) {\footnotesize$\bm{\mathcal{R}}_3$}; \\

          \node[process] (X0){\footnotesize$\bm{\mathcal{X}}_0$}; 
    \&   \node[source] (M1) {\footnotesize$\bm{\mathcal{M}}_1$};
    \&   \node[source] (M2) {\footnotesize$\bm{\mathcal{M}}_2$};
    \&   \node[process] (X1){\footnotesize$(\widehat{\bm{\mathcal{Y}}}_{1})$\\ \footnotesize$\bm{\mathcal{X}}_1$}; 
    \&   \node[source] (M3) {\footnotesize$\bm{\mathcal{M}}_3$};
    \&   \node[source] (M4) {\footnotesize$\bm{\mathcal{M}}_4$};
    \&   \node[process] (X2){\footnotesize$(\widehat{\bm{\mathcal{Y}}}_{2})$\\ \footnotesize $\bm{\mathcal{X}}_2$};
    \&   \node[source] (Space1) {};
    \&   \node[source] (Space2) {};  \\

    \&   \node[process] (LS1){\footnotesize Problem (\ref{2.9})}; 
     \&
      \&
       \& \node[process] (LS2){\footnotesize Problem (\ref{2.9})};
        \&
        \& \node[source] (s){\red{Set $s=2$}}; \\

           \node[sink] (theta101) {$\alpha_{1,0}^1$};   
     \&    \node[sink] (theta111) {$\alpha_{1,1}^1$};
     \&  \node[sink] (theta121) {$\alpha_{1,2}^1$};
        \&    \node[source] (theta102) {$\alpha_{1,0}^2$};
     \&    \node[source] (theta112) {$\alpha_{1,1}^2$};
    \&   \node[source] (theta122) {$\alpha_{1,2}^2$};
       \\ };

 \draw[to,red,thick] (R1) --(M1);        
 \draw[to,red,thick] (R1) --(M2);    
 \draw[to,red,thick] (X0) --(LS1);        
 \draw[to,red,thick] (M1) --(LS1);        
 \draw[to,red,thick] (M2) --(LS1); 
 \draw[to,red,dashed] (LS1) --(theta101);
 \draw[to,red,dashed] (LS1) --(theta111);        
 \draw[to,red,dashed] (LS1) --(theta121);      
 \draw[to,red,thick] (LS1) --(X1);  
 \draw[to,red,thick] (X1) --(R2);

 \draw[to,blue,thick] (R2) --(M3);        
 \draw[to,blue,thick] (R2) --(M4);             
 \draw[to,blue,thick] (X1) --(LS2);     
 \draw[to,blue,thick] (M3) --(LS2);        
 \draw[to,blue,thick] (M4) --(LS2); 
 \draw[to,blue,dashed] (LS2) --(theta102);
 \draw[to,blue,dashed] (LS2) --(theta112);        
 \draw[to,blue,dashed] (LS2) --(theta122);     
   
 \draw[to,blue,thick] (LS2) --(X2);  
 \draw[to,blue,thick] (X2) --(R3);

 \draw[to,thick] (R3) --(Space1);        
 \draw[to,thick] (R3) --(Space2);

\end{tikzpicture}

\end{center}

\caption{Economic ELRAP4TS Algorithm (With $r=4$ and $s=2$)}\label{Figure 3}
\end{figure}






\section{TRIP}
\begin{dfn}\label{definition 1}\cite{ZWHWW2019}
\textbf{(TRIP).} Let the linear operator $\Phi : \mathbb{R}^{n_1 \times n_2 \times n_3} \rightarrow \mathbb{R}^m$ be a linear map on the linear space of tensors of size $n_1 \times n_2 \times n_3$ with $n_1 \leq n_2$. Then for the t-SVD decomposition and every integer $r$ with $1 \leq r \leq n_1$, the TRIP constant $\delta_r$ of $\Phi$ is the smallest quantity such that
$$(1-\delta_r) ||\bm{\mathcal{X}}||^2_F \leq ||\Phi(\bm{\mathcal{X}})||^2_2 \leq (1+\delta_r) ||\bm{\mathcal{X}}||^2_F$$  
 for all tensors $\bm{\mathcal{X}} \in \mathbb{R}^{n_1 \times n_2 \times n_3}$ of rank at most $r$.
\end{dfn}

\begin{theorem}\label{Theorem 2}
For $\delta, \epsilon \in (0,1)$, a random draw of an L-subgaussian measurement ensemble $\Phi : \mathbb{R}^{n_1 \times n_2 \times n_3} \rightarrow \mathbb{R}^m$ satisfies $\delta_r \leq \delta$ with probability at least $1-\epsilon$, provided that 
$$m \geq C \alpha^2 \delta^{-2} \max \{(r \cdot r \cdot n_3+ n_1 \cdot r \cdot n_3  + n_2 \cdot r \cdot n_3) \log(2), \log( \eta^{-1}) \},$$
where $r$ is tensor tubal rank. 
The constants $C > 0$ only depend on the subgaussian parameter $L$.
\end{theorem}

To prepare for the proof of  Theorem \ref{Theorem 2}, we state several useful Lemmas. 
The proof of Theorem \ref{Theorem 2} also uses $\epsilon$-nets and covering numbers, see \cite{V2012} for background on these topics.

\begin{dfn}\label{Definition 2}\cite{V2012}
\textbf{(Nets, covering numbers)}
A set $\mathscr{N}^{\mathscr{X}}_{\epsilon} \subset \mathscr{X}$ with $\mathscr{X}$ a subset of a normed space is called an $\epsilon$-net of $\mathscr{X}$ with respect to the norm $||\cdot ||$ if for each $\mathbf v \in  \mathscr{X}$ there exists $\mathbf v_0 \in \mathscr{N}^{\mathscr{X}}_{\epsilon}$ with $||\mathbf v_0 - \mathbf v || \leq \epsilon$. The minimal cardinality of an $\epsilon$-net of $\mathscr{X}$ with respect to the norm $||\cdot ||$ is denoted by $\mathscr{N} (\mathscr{X}, ||\cdot ||, \epsilon) $ and called the covering number of  $\mathscr{X}$ (at scale $\epsilon$).
\end{dfn}

Equivalently, $\mathscr{N} (\mathscr{X}, ||\cdot ||, \epsilon) $ is the minimal number of balls with radii $\epsilon$ and with centers in $\mathscr{X}$ needed to cover $\mathscr{X}$. The following is well-known \cite{V2012, RSS2017} and will be used frequently in what follows.

\begin{lemma}\label{Lemma 1}\cite{V2012}
	\textbf{(Covering numbers of the sphere)}
Let $\mathscr{X}$ be a subset of a vector space of real dimension $k$ with norm $||\cdot ||$, and let $0 < \epsilon < 1$. Then there exists an $\epsilon$-net 
$\mathscr{N}^{\mathscr{X}}_{\epsilon} \subset \mathscr{X}$ with $$|\mathscr{N}^{\mathscr{X}}_{\epsilon}| \leq {Vol(\mathscr{X} + {\epsilon \over 2} \mathscr{B}) \over Vol({\epsilon \over 2} \mathscr{B})},$$ where ${\epsilon \over 2} \mathscr{B}$ is an $\epsilon / 2$ ball with respect to the norm $||\cdot ||$ and $$\mathscr{X} + {\epsilon \over 2} \mathscr{B} = \{\mathbf x+\mathbf y : \mathbf x \in  \mathscr{X}, \mathbf y \in {\epsilon \over 2} \mathscr{B}\}.$$ Specifically if $\mathscr{X}$ is a subset of  the $||\cdot ||$-unit ball then $\mathscr{X} + {\epsilon \over 2} \mathscr{B}$ is contained in the $(1+{\epsilon \over 2})$-ball and thus
$$|\mathscr{N}^{\mathscr{X}}_{\epsilon}| \leq {(1 + \epsilon / 2)^k \over (\epsilon / 2)^k} = (1 + {2 \over \epsilon})^k < (3 / \epsilon)^k.$$
\end{lemma}

The covering number Lemmas are  critical for the proof of Theorem \ref{Theorem 2} where we need to calculate the covering number for the set of rank $r$ tensors with unit Frobenius norm.

\begin{lemma}\label{Lemma 2}
(\textbf{Covering numbers related to the t-SVD}). The covering numbers of 
$$\mathscr{S}_r = \{\bm{\mathcal{X}} \in \mathbb{R}^{n_1 \times n_2 \times n_3} : {\rm rank_t} (\bm{\mathcal{X}}) \leq r, ||\bm{\mathcal{X}}||_F = 1  \}$$
with respect to the Frobenius norm are bounded by
\begin{equation}\label{47}
\mathscr{N} (\mathscr{S}_r, ||\cdot ||_F, \epsilon) \leq (9 / \epsilon)^{r \cdot r \cdot n_3+ n_1 \cdot r \cdot n_3  + n_2 \cdot r \cdot n_3}. 
\end{equation}
\end{lemma}


\begin{proof}
 The proof follows a same strategy as the Lemma 3.1 of \cite{CP2011} and Lemma 3 of \cite{RSS2017}. 
 The (truncating) t-SVD decomposition $\bm{\mathcal{X}} = \bm{\mathcal{U}} * \bm{\mathcal{S}} * \bm{\mathcal{V}}^T$ of any $\bm{\mathcal{X}} \in \mathscr{S}_r$ obeys $|| \bm{\mathcal{S}}||_F = 1$ ($ \bm{\mathcal{S}}$ is a f-diagonal tensor with singular tubes on the diagonal, and $ \bm{\mathcal{U}}$ and  $\bm{\mathcal{V}}$ are partially orthogonal tensors of left- and right-singular vectors of tubes). We constructs an $\epsilon$-net for $ \mathscr{S}_r$ by covering the sets of tensors $ \bm{\mathcal{U}}$, $ \bm{\mathcal{V}}$ with partially orthogonal lateral slices and the set of unit Frobenius tensors $ \bm{\mathcal{S}}$. 

Let $\mathscr{D}$ be the set of f-diagonal tensors $\bm{\mathcal{X}} \in \mathbb{R}^{r \times r  \times n_3}$ with 
unit Frobenius norm, which is included in $\mathscr{F} = \{ \bm{\mathcal{X}} \in \mathbb{R}^{r \times r  \times n_3} : ||\bm{\mathcal{X}} ||_F = 1\}$. Hence Lemma \ref{Lemma 1} gives an $\epsilon / 3$-net $\mathscr{N}^{\mathscr{F}}_{\epsilon / (d + 1)}$ in respect of the Frobenius norm of cardinality $$|\mathscr{N}^{\mathscr{F}}_{\epsilon / (d + 1)} | \leq (9 /\epsilon)^{r \cdot r \cdot n_3}.$$ For covering $\mathscr{O}_{n_1,r,n_3} = \{\bm{\mathcal{U}} \in \mathbb{R}^{n_1 \times r \times n_3}: \bm{\mathcal{U}}^T * \bm{\mathcal{U}} = \bm{\mathcal{I}} \}$ and $\mathscr{O}_{n_2,r,n_3} = \{\bm{\mathcal{V}} \in \mathbb{R}^{n_2 \times r \times n_3}: \bm{\mathcal{V}}^T * \bm{\mathcal{V}} = \bm{\mathcal{I}} \}$, it is crucial to use the norm $|| \cdot ||_{1,F}$, which is defined as 
$$ || \bm{\mathcal{Y}} ||_{1,F} =  \max_{i,j} || \hat{\bm{\mathcal{Y}}} (:,j,k) ||_F, $$
where $\hat{\bm{\mathcal{Y}}} (:,j,k)$ denotes the $j,k$-th tube fiber in $\hat{\bm{\mathcal{Y}}}$ from $\bm{\mathcal{Y}}$ using the DFT. Obviously, we have that $ || \bm{\mathcal{Y}} ||_{1,F}=|| \overline{\mathbf Y} ||_{1,F}.$

Because the elements of $\mathscr{O}_{n_1,r,n_3}$ have normed lateral slices, it holds $\mathscr{O}_{n_1,r,n_3} \subset \mathscr{Q}_{n_1,r,n_3} = \{ \bm{\mathcal{Y}} \in \mathbb{R}^{n_1 \times r \times n_3}:  || \bm{\mathcal{Y}} ||_{1,F} \leq 1 \}$. Lemma \ref{Lemma 1} provides 
$$\mathscr{N} (\mathscr{O}_{n_1,r,n_3}, ||\cdot ||_{1,F}, \epsilon / 3) \leq 9 / \epsilon)^{n_1 \cdot r \cdot n_3 },$$
 i.e. there exists an $ \epsilon / 3$-net $\mathscr{N}^{\mathscr{O}_{n_1,r,n_3}}_{\epsilon / 3}$ of this cardinality.

 Then the set
 $$\mathscr{N}^{\mathscr{S}_r}_{\epsilon} := \{ \tilde{\bm{\mathcal{U}}} * \tilde{\bm{\mathcal{S}}} * \tilde{\bm{\mathcal{V}}}^T :  \tilde{\bm{\mathcal{S}}} \in \mathscr{N}^{\mathscr{D}}_{\epsilon / 3}, \ \tilde{\bm{\mathcal{U}}} \in \mathscr{N}^{\mathscr{O}_{n_1,r,n_3}}_{\epsilon / 3}, $$ 
 $$ \ and \ \ \tilde{\bm{\mathcal{V}}} \in \mathscr{N}^{\mathscr{O}_{n_2,r,n_3}}_{\epsilon / 3} \} $$
obeys
$$|\mathscr{N}^{\mathscr{S}_r}_{\epsilon} | \leq \mathscr{N} (\mathscr{D}, ||\cdot ||_{F}, \epsilon / 3) \cdot \mathscr{N} (\mathscr{O}_{n_1,r,n_3}, ||\cdot ||_{1,F}, \epsilon / 3) \cdot$$
 $$\cdot \mathscr{N} (\mathscr{O}_{n_2,r,n_3}, ||\cdot ||_{1,F}, \epsilon / 3) \leq  (9 / \epsilon)^{r \cdot r \cdot n_3+ n_1 \cdot r \cdot n_3  + n_2 \cdot r \cdot n_3}.$$
 It remains to show that $\mathscr{N}^{\mathscr{S}_r}_{\epsilon}$ is an $\epsilon$-net for $\mathscr{S}_r$, i.e. that for all $\bm{\mathcal{X}} \in \mathscr{S}_r$ there exists $\tilde{\bm{\mathcal{X}}} \in\mathscr{N}^{\mathscr{S}_r}_{\epsilon}$ with $||\bm{\mathcal{X}} - \tilde{\bm{\mathcal{X}}}||_F \leq \epsilon$. 
 
 At last, we fix $\bm{\mathcal{X}} \in \mathscr{S}_r$ and decompose $\bm{\mathcal{X}}$ as 
 $\bm{\mathcal{X}} = \bm{\mathcal{U}} * \bm{\mathcal{S}} * \bm{\mathcal{V}}^T$. Then there exists $\tilde{\bm{\mathcal{X}}} = \tilde{\bm{\mathcal{U}}} * \tilde{\bm{\mathcal{S}}} * \tilde{\bm{\mathcal{V}}}^T \in \mathscr{N}^{\mathscr{S}_r}_{\epsilon}$ with  $\tilde{\bm{\mathcal{U}}} \in \mathscr{N}^{\mathscr{O}_{n_1,r,n_3}}_{\epsilon / 3}$, $\tilde{\bm{\mathcal{V}}} \in \mathscr{N}^{\mathscr{O}_{n_2,r,n_3}}_{\epsilon / 3}$ and $\tilde{\bm{\mathcal{S}}} \in \mathscr{N}^{\mathscr{D}}_{\epsilon / 3}$ obeying $||\bm{\mathcal{U}} - \tilde{\bm{\mathcal{U}}}||_{1,F} \leq \epsilon / 3$, $||\bm{\mathcal{V}} - \tilde{\bm{\mathcal{V}}}||_{1,F} \leq \epsilon / 3$ and $||\bm{\mathcal{S}} - \tilde{\bm{\mathcal{S}}}||_F \leq \epsilon / 3$. This gives 
\begin{equation}\label{48}
\begin{array}{lll}
||\bm{\mathcal{X}} - \tilde{\bm{\mathcal{X}}}||_F = || \bm{\mathcal{U}} * \bm{\mathcal{S}} * \bm{\mathcal{V}}^T- \tilde{\bm{\mathcal{U}}} * \tilde{\bm{\mathcal{S}}} * \tilde{\bm{\mathcal{V}}}^T||_F \\
 \ \ \ \ \leq || (\bm{\mathcal{U}} - \tilde{\bm{\mathcal{U}}})  * \bm{\mathcal{S}} * \bm{\mathcal{V}}^T||_F + || \tilde{\bm{\mathcal{U}}} *( \bm{\mathcal{S}} - \tilde{\bm{\mathcal{S}}}) * \bm{\mathcal{V}}^T||_F \\
 \ \ \ \ \ \ \ + ||\tilde{\bm{\mathcal{U}}} * \tilde{\bm{\mathcal{S}}} * (\bm{\mathcal{V}} - \tilde{\bm{\mathcal{V}}})^T||_F 
\end{array}
\end{equation}
 
For the first term, since $\bm{\mathcal{V}}$ is an partially orthogonal tensor and have unitary invariance \cite{KM2011}, 
$ || (\bm{\mathcal{U}} - \tilde{\bm{\mathcal{U}}})  * \bm{\mathcal{S}} * \bm{\mathcal{V}}^T||_F = || (\bm{\mathcal{U}} - \tilde{\bm{\mathcal{U}}})  * \bm{\mathcal{S}}||_F$, and 
\begin{equation*}
\begin{array}{lll}
|| (\bm{\mathcal{U}} - \tilde{\bm{\mathcal{U}}})  * \bm{\mathcal{S}}||^2_F =|| \bm{\mathcal{C}}  * \bm{\mathcal{S}}||^2_F \\
\qquad = {1 \over n_3}|| \overline{\mathbf C}  * \overline{\mathbf S}||^2_F  \leq  {1 \over n_3}||\overline{\mathbf S}||^2_F \cdot || \overline{\mathbf C} ||_{1,F}   \\
\qquad  = ||\bm{\mathcal{S}}||^2_F \cdot || \overline{\mathbf C} ||_{1,F} = ||\bm{\mathcal{S}}||^2_F \cdot || \bm{\mathcal{C}} ||_{1,F}  \\
\qquad  =  ||\bm{\mathcal{S}}||^2_F \cdot || (\bm{\mathcal{U}} - \tilde{\bm{\mathcal{U}}})  ||^2_{1,F} \\
\qquad  \leq  (\epsilon / 3)^2,
\end{array}
\end{equation*}
where $\bm{\mathcal{C}} = (\bm{\mathcal{U}} - \tilde{\bm{\mathcal{U}}})$. Hence, $|| (\bm{\mathcal{U}} - \tilde{\bm{\mathcal{U}}})  * \bm{\mathcal{S}} * \bm{\mathcal{V}}^T||_F \leq  \epsilon / 3$. The same argument gives $||\tilde{\bm{\mathcal{U}}} * \tilde{\bm{\mathcal{S}}} * (\bm{\mathcal{V}} - \tilde{\bm{\mathcal{V}}})^T||_F \leq  \epsilon / 3$. To bound the middle term, observe that $|| \tilde{\bm{\mathcal{U}}} *( \bm{\mathcal{S}} - \tilde{\bm{\mathcal{S}}}) * \bm{\mathcal{V}}^T||_F = ||  \bm{\mathcal{S}} - \tilde{\bm{\mathcal{S}}}||_F \leq  \epsilon / 3$. Therefore, we have the desired result. 
 \end{proof}


The proof of Theorem \ref{Theorem 2} also requires the following result from Corollary 5.4 of \cite{D2016}.

\begin{cor}\label{Corollary 5.4} \cite{D2016}
Let $\mathscr{S}_{(1)}, \cdots , \mathscr{S}_{(k)}$ be subsets of a Hilbert space $H$ and let $\mathscr{S} = \cup^k_{i=1} \mathscr{S}_{(i)}$.
Set  $$\mathscr{S}_{(i),nv} = \{\mathbf x / ||\mathbf x ||_2 : \mathbf x \in \mathscr{S}_{(i)} \}.$$
suppose that $\mathscr{S}_{(i),nv}$ has covering dimension $K_i$ with parameter $c_i$ and base covering $\mathscr{N}_{0,(i)}$ with respect to $d_H$. Set $K = \max_i  K_i, c = \max_i c_i$ and $\mathscr{N}_0 = \max_i  \mathscr{N}_{0,(i)}$. Let $\Phi : \Omega \times  H \rightarrow  \mathbb{R}^m$ be a subgaussian map on $\mathscr{S}_{nv}$. Then, for any $0 < \delta, \eta < 1$, we have $\mathbb{P}(\delta_{\mathscr{S},\Phi} \leq \eta)$, provided that 
$$m \geq C \alpha^2 \delta^{-2} max \{\log k + \log \mathscr{N}_0 + K \log(c), \log( \eta^{-1)} \}.$$ 
\end{cor}

Proof of Theorem \ref{Theorem 2}. The proof follows the same strategy as that of Example 5.8 of \cite{D2016}. We consider the Frobenius inner product 
$ \langle \bm{\mathcal{X}}, \bm{\mathcal{Y}} \rangle$, the corresponding norm $||\bm{\mathcal{X}}||_F = \langle \bm{\mathcal{X}}, \bm{\mathcal{Y}} \rangle^2$  and the induced metric $d_F(\bm{\mathcal{X}}, \bm{\mathcal{Y}}) = ||\bm{\mathcal{X}} - \bm{\mathcal{Y}}||_F$.

As said earlier,
$$\mathscr{S}_r = \{\bm{\mathcal{X}} \in \mathbb{R}^{n_1 \times n_2 \times n_3} : {\rm rank_t} (\bm{\mathcal{X}}) \leq r, ||\bm{\mathcal{X}}||_F = 1  \},$$ then $\delta_r = \delta_{\mathscr{S}_r, \Phi} $.
It is shown in Lemma {\ref{Lemma 2}} that for any $0 < \epsilon \leq 1$, the covering number is 
$$\mathscr{N} (\mathscr{S}_r, ||\cdot ||_F, \epsilon) \leq (3(2+1) / \epsilon)^{r \cdot r \cdot n_3+ n_1 \cdot r \cdot n_3  + n_2 \cdot r \cdot n_3}.$$
In other words, $\mathscr{S}_r $ has covering dimension $K = r \cdot r \cdot n_3+ n_1 \cdot r \cdot n_3  + n_2 \cdot r \cdot n_3$ with parameter $c = 9$. Corollary \ref{Corollary 5.4} means that for any subgaussian map $\Phi$ and $0< \delta, \eta < 1$, we have $\mathbb{P}(\delta_r \geq \delta) \leq \eta$, provided that 
$$m \geq C \alpha^2 \delta^{-2} \max \{(r \cdot r \cdot n_3+ n_1 \cdot r \cdot n_3  + n_2 \cdot r \cdot n_3) \log(2), \log( \eta^{-1}) \}.$$ 
This completes the proof. $\Box$



\section{Convergence Analysis}

We will proof that Algorithm \ref{alg:Framwork} converges linearly in this section. This  is shown in Theorem \ref{Theorem 3.1}. 





\begin{theorem}\label{Theorem 3.1}
	Algorithm \ref{alg:Framwork} in the standard LRAP4TS and economic LRAP4TS versions both have the linear convergence rate of
	\begin{equation*}
	\left\|\bm{\mathcal{R}}_k\right\| \leqslant  \left( {\sqrt{1-\frac{1}{\min(m,n)}}} \right)^{k-1} \cdot \left\|\Phi^{-1}(\mathbf{b})\right\| \text{ for all } k\geqslant 1,
	\end{equation*}
	where $\mathbf{b}=\Phi(\bm{\mathcal{Y}})=\phi \cdot vec(\bm{\mathcal{Y}})$.  This holds for all tensors $\bm{\mathcal{Y}}$ of rank at most $r$.
\end{theorem}

\proof First we give that both LRAP4TS and ELRAP4TS satisfy the following inequality:
\begin{equation*}
\|\bm{\mathcal{R}}_{k+1}\|^2 \leqslant \|\bm{\mathcal{R}}_{k}\|^2-\langle {\bm{\mathcal{M}}_{k,1}, \bm{\mathcal{R}}_k} \rangle^2.
\end{equation*}
For the LRAP4TC algorithm we have
\begin{equation*}
\begin{array}{lll}
\|\bm{\mathcal{R}}_{k+1} \|^2 &=& \mathop{\min}\limits_{\bm{\theta} \in \mathbb{R}^{sk}}\| \Phi^{-1} \Phi(\bm{\mathcal{Y}})-\sum_{i=1}^{k}\sum_{j=1}^{s}\theta_{i,j} \Phi^{-1} \Phi(\bm{\mathcal{M}}_{i,j}) \|^2 \\ 
&\leqslant& \mathop{\min}\limits_{\theta_{k,1}\in \mathbb{R}}\| \Phi^{-1} \Phi(\bm{\mathcal{Y}})- \bm{\mathcal{X}}_{k-1}- \theta_{k,1} \Phi^{-1} \Phi(\bm{\mathcal{M}}_{k,1}) \|^2  \\
&=& \mathop{\min}\limits_{\theta_{k,1}\in \mathbb{R}}\| \bm{\mathcal{R}}_k- \theta_{k,1} \Phi^{-1} \Phi(\bm{\mathcal{M}}_{k,1}) \|^2. \\
\end{array}
\end{equation*}
And for the  ELRAP4TS algorithm we have
\begin{equation*}
\begin{array}{lll}
\|\bm{\mathcal{R}}_{k+1} \|^2 = \\
=  \mathop{\min}\limits_{\bm{\alpha}\in \mathbb{R}^{s+1}}\| \Phi(\bm{\mathcal{Y}})- \alpha_0 \bm{\mathcal{X}}_{k-1} - \alpha_1 \Phi(\bm{\mathcal{M}}_{k,1}) -\cdots-\alpha_s \Phi(\bm{\mathcal{M}}_{k,s})  \|_\Omega^2 \\ 
  \leqslant \mathop{\min}\limits_{\alpha_1\in \mathbb{R}}\| \Phi^{-1} \Phi(\bm{\mathcal{Y}})- \bm{\mathcal{X}}_{k-1} - \alpha_1 \Phi^{-1} \Phi(\bm{\mathcal{M}}_{k,1}) \|^2  \\
= \mathop{\min}\limits_{\alpha_1\in \mathbb{R}}\|\bm{\mathcal{R}}_{k} - \alpha_1 \Phi^{-1} \Phi(\bm{\mathcal{M}}_{k,1}) \|^2 .
\\
\end{array}
\end{equation*}
In both cases we get closed form solutions for  $\theta_{k,1}^*=\frac{\langle \bm{\mathcal{R}}_{k},\Phi^{-1} \Phi(\bm{\mathcal{M}}_{k,1})\rangle}{\langle \Phi^{-1} \Phi(\bm{\mathcal{M}}_{k,1}), \Phi^{-1} \Phi(\bm{\mathcal{M}}_{k,1})\rangle}=\frac{\langle \bm{\mathcal{R}}_{k}, \bm{\mathcal{M}}_{k,1}\rangle}{\langle \Phi^{-1} \Phi(\bm{\mathcal{M}}_{k,1}), \Phi^{-1} \Phi(\bm{\mathcal{M}}_{k,1})\rangle}$ and $\alpha_1^*=\frac{\langle \bm{\mathcal{R}}_{k}, \Phi^{-1} \Phi(\bm{\mathcal{M}}_{k,1})\rangle}{\langle \Phi^{-1} \Phi(\bm{\mathcal{M}}_{k,1}), \Phi^{-1} \Phi(\bm{\mathcal{M}}_{k,1})\rangle}=\frac{\langle \bm{\mathcal{R}}_{k}, \bm{\mathcal{M}}_{k,1}\rangle}{\langle \Phi^{-1} \Phi(\bm{\mathcal{M}}_{k,1}), \Phi^{-1} \Phi(\bm{\mathcal{M}}_{k,1})\rangle}$. Plugging the optima $\theta_{k,1}^*$ and $\alpha_1^*$ back into the above formulas, we get
\begin{equation}\label{2.7}
\begin{array}{lll}
\|\bm{\mathcal{R}}_{k+1} \|^2 &\leqslant&\| \bm{\mathcal{R}}_k-\frac{\langle \bm{\mathcal{R}}_{k}, \bm{\mathcal{M}}_{k,1}\rangle}{\langle \Phi^{-1} \Phi(\bm{\mathcal{M}}_{k,1}), \Phi^{-1} \Phi(\bm{\mathcal{M}}_{k,1})\rangle} \Phi^{-1} \Phi(\bm{\mathcal{M}}_{k,1}) \|^2 \\ 
&=&  \| \bm{\mathcal{R}}_k\|^2-\frac{\langle \bm{\mathcal{R}}_{k}, \bm{\mathcal{M}}_{k,1}\rangle^2}{\langle \Phi^{-1} \Phi(\bm{\mathcal{M}}_{k,1}), \Phi^{-1} \Phi(\bm{\mathcal{M}}_{k,1})\rangle} \\ 
&=& \| \bm{\mathcal{R}}_k\|^2-\langle \bm{\mathcal{R}}_{k}, \bm{\mathcal{M}}_{k,1}\rangle^2,\\ 
\end{array}
\end{equation}
since $\langle \Phi^{-1} \Phi(\bm{\mathcal{M}}_{k,1}), \Phi^{-1} \Phi(\bm{\mathcal{M}}_{k,1})\rangle \leqslant 1$ from $\Phi \Phi^{-1}$ is an identity operator and $\Phi = \phi \cdot vec(\cdot)$.

Using  inequality (\ref{2.7}) and Lemma \ref{Lemma 2.2}, we obtain that
$$\|\bm{\mathcal{R}}_{k+1}\|^2\leqslant \|\bm{\mathcal{R}}_{k}\|^2-\langle \bm{\mathcal{R}}_{k}, \bm{\mathcal{M}}_{k,1}\rangle^2 \leqslant \left(1-\frac{1}{min(m,n)}\right)\|\bm{\mathcal{R}}_{k}\|^2.$$
This completes the proof.  $\Box$

Based on Theorem \ref{Theorem 2} some linear operators  $\Phi$ satisfies the TRIP condition with high probability. By assuming that the  TRIP condition holds, we prove the following approximation result.

\begin{theorem}\label{Theorem B1}
Let $\bm{\mathcal{Y}}$ be a tensor of tensor tubal rank $r$. Suppose the measurement mapping $\Phi(\bm{\mathcal{X}})$ satisfies TRIP for rank-$r_0$ with $\delta_{r_0}=\delta_{r_0}(\Phi)<1$ with $r_0\geqslant 2r$. The output tensor $\bm{\mathcal{M}}(\bm{\theta}^k)$ approximates the exact tensor $\bm{\mathcal{Y}}$ in the following sense: there is a positive constant $\tau$ such that
\begin{equation*}
\|\bm{\mathcal{M}}(\bm{\theta}^k)-Y\|_F\leqslant \frac{C}{\sqrt{1-\delta_{r_0}}} \tau^k
\end{equation*}
for all $k=1,\cdots,\lfloor{(r_0-r)/s}\rfloor$, where $\lfloor{(r_0-r)/s}\rfloor$ denotes the largest inter less than or equal to $r/s$ and $C>0$ is a constant depending on $\Phi$.  
\end{theorem}

\proof

Based on the definition of $\delta_{r_0}$, for $s \red{\cdot} k+r\leqslant r_0$, we have
\begin{equation*}
\begin{array}{lll}
(1-\delta_{r_0})\|\bm{\mathcal{M}}(\bm{\theta}^k)-\bm{\mathcal{Y}}\|^2_F &\leqslant&   \|\Phi(\bm{\mathcal{M}}(\bm{\theta}^k))-\Phi(\bm{\mathcal{Y}})\|^2_2   \\
       &=&    \|\Phi(\bf{\mathcal{R}}_k)\|^2_2=\|\phi \cdot vec(\bm{\mathcal{R}}_k)\|^2_2\\
       &\leqslant&\|\phi\|^2_2\|vec(\bm{\mathcal{R}}_k)\|^2_2=\|\phi\|^2_2\|\bm{\mathcal{R}}_k\|^2_F\\
       &\leqslant& \|\phi \|^2_2\tau^{2k}\|\Phi^{-1}(\mathbf{b})\|^2_F
       \end{array}
\end{equation*}
where $\tau=\sqrt{1-\frac{1}{\min(m,n)}}$ from Theorem \ref{Theorem 3.1}. So we have
\begin{equation*}
\|\bm{\mathcal{M}}(\bm{\theta}_k)-\bm{\mathcal{Y}}\|^2_F  \leqslant\frac{ \|\phi \|^2_2\tau^{2k}}{1-\delta_{r_0}}\|\Phi^{-1}(\mathbf{b})\|^2_F. 
\end{equation*}
Therefore, we have the desired result. $\Box$

The above convergence result require $s \cdot k+r\leqslant r_0$, which guarantees the TRIP condition for all estimated tensors during the iteration process.

\section{Numerical Tests and Applications }

All our methods have been implemented for tensor completion in MATLAB. All test experiments were run on a Macbook Pro with OSX High Sierra, an Intel Core i5 2.6 GHz processor and with 8G RAM. We have noted that the speed of the RAM is very important for Matlab speeds. Our experimental results are compared with respect to the root-mean-square error (RMSE), defined as $RMSE = \sqrt{|| \bm{\mathcal{X}} - \bm{\mathcal{Y}} ||^2_F / (n_1 n_2 n_3) }$. The proposed ELRAP4TC method reduces the CPU time significantly from that of LRAP4TC with a small decrease in quality.
The test results with ELRAP4TC are only  included for clarity. The t-SVD algorithm  to find the truncated t-SVD in the ELRAP4TC algorithm. For further speed up significantly,  readers might adopt the rt-SVD method instead. 

\subsection{ Video Recovery}

 A video clip, shot with a 1/50 sec or faster shutter speed every 1/25 sec contains 25 digital data frames for every second of real time.
In our video recovery example below, we use a $144 \times 256 \times 40$  black and white gray-scale "basketball" video. Our video clip depicts 1.6 sec of a game and contains  40 digital images in avi format. Each frame holds 144 by 256 pixels of raw  sensor data. To visualize, consider these data frames vertically,  40 in total, one behind the other along the time line, capturing 1.6 second of time. This visualization helps us to interpret the data as a 3-dimensional tensor. Such a video generated  tensor will have low tubal rank because  if one were to look perpendicularly to the video image data planes in the direction of time over a small area of the images themselves. Then one would (locally) see almost identical images with little changes in any one small tube. 
 Looking  at the same small cut out area for all 40 consecutive video frames, there will generally be only small changes because the angle of view does not change all that quickly, fixed objects do not move, people do not run that fast, not in a low resolution video and not over one second or two. 
Hence if some of the 40 frames are missing or damaged -- as we assume -- we would start from a  low tubal rank tensor and try  to reconstruct the video on that premise. We have set rank $r=100$ in our first experiment, i.e., we retain $50\%$ of the given pixels as known entries in $\Omega$. The recovery of the damaged video  has thus  become a tensor completion problem that we can solve  with our ELRAP4TC algorithm.  Since ELRAP4TC is always applied with $s \leqslant r$  we   work with $s=1, 2, 3$ here.

In Figure \ref{Figure 4}  the convergence characteristics are shown for the ELRAP4TC algorithm. The running times decrease in the ELRAP4TC method for increasing values of $s$.

\begin{figure}[h]
	\begin{center}
	\includegraphics[scale=0.45]{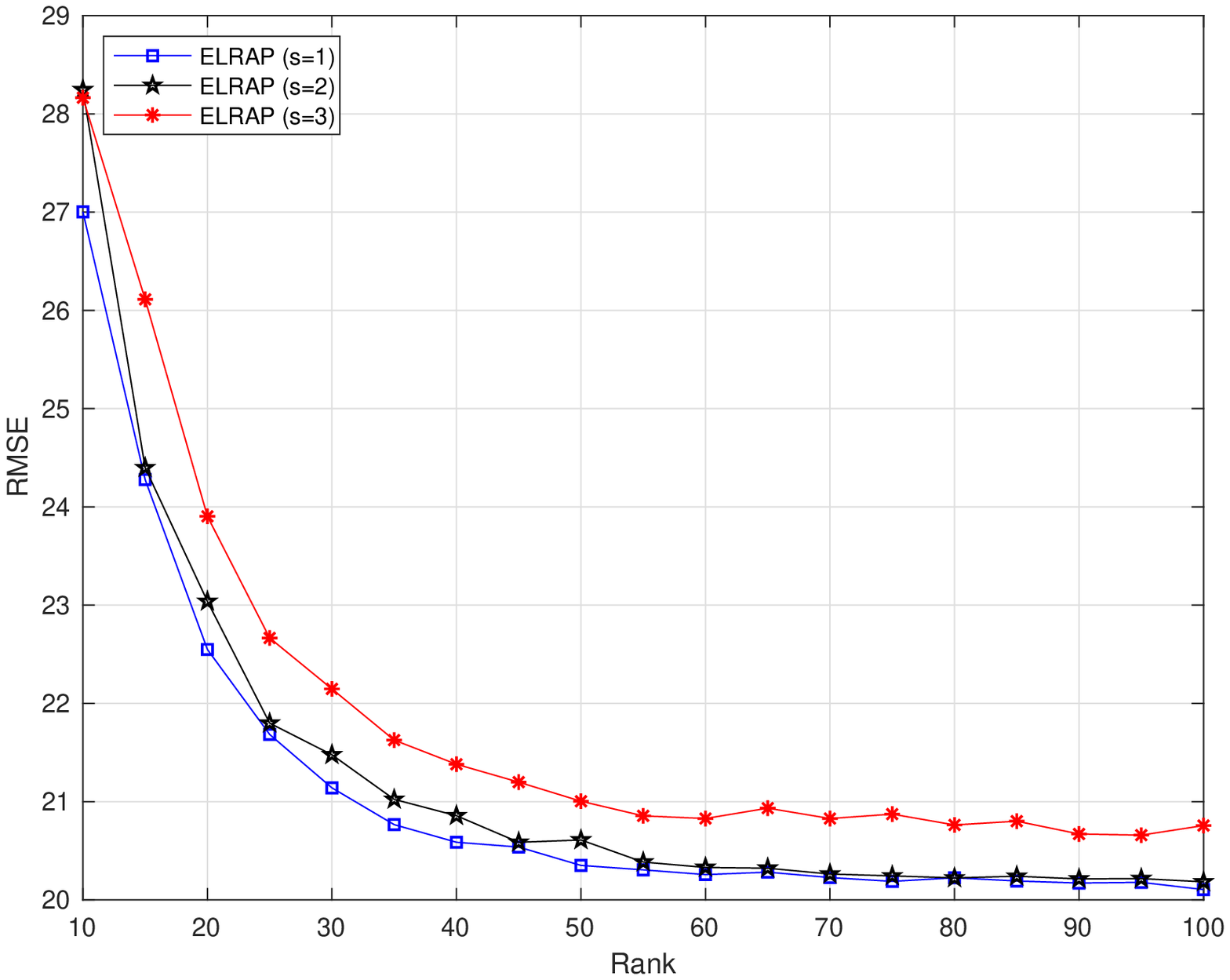}
		\includegraphics[scale=0.45]{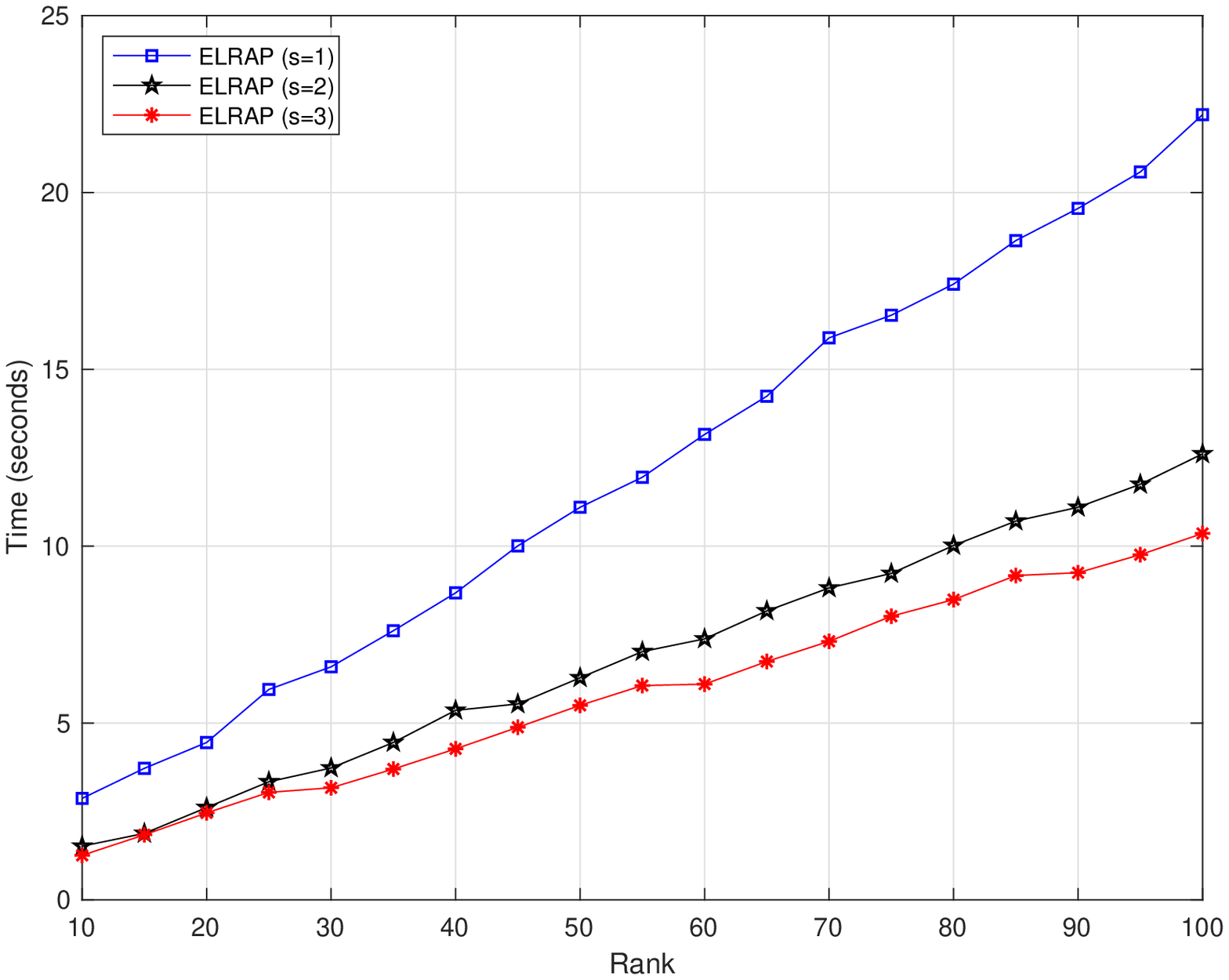}
		\caption{Linear convergence of ELRAP4TC with $s=1, 2, 3$ (on the left) and  run times  in seconds (on the right)}
		\label{Figure 4}
	\end{center}
\end{figure}

To evaluate our LRAP4TC we  compare it  with two state-of-the art algorithms such as HoMP \cite{YMS2015} and ADMM with the t-SVD as subroutine (ADMM-t-SVD) \cite{ZA2017}. The two previous algorithm own the theoretical recovery guarantee. For ADMM-t-SVD, the parameter $\lambda$  is set to $\lambda = 1/ \sqrt{3 \max{(n_1,n_2)}}$. For HoMP, we empirically set $r =100$ which is the same with our algorithm. From the numerical experiments in Section V of \cite{YMS2015},  the HoMP algorithm is the fastest state-of-the-art algorithms four years ago for the tensor completion problem. The comparison result is shown in Figure \ref{Figure 5}.

\begin{figure}[h]
	\begin{center}
\includegraphics[scale=0.36]{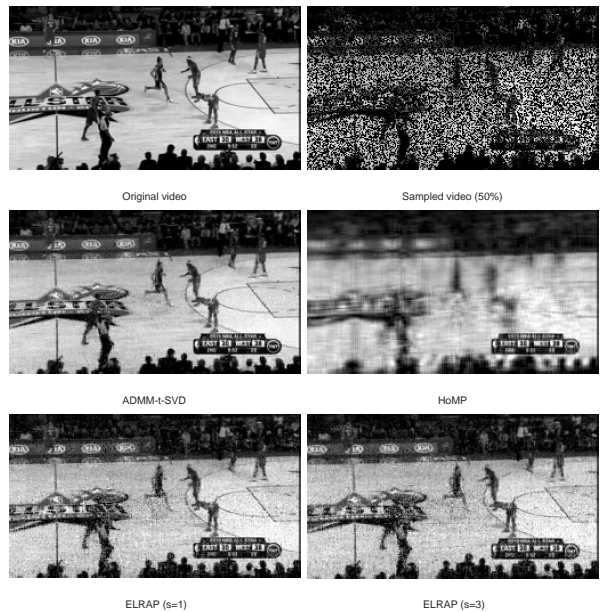}
		\caption{The 30th frame of completion result for a basketball video.}
		\label{Figure 5}
	\end{center}
\end{figure}

We measure  the running time and  RMSE for all three methods. Running time and RMSE are listed  in Table \ref{Table 2}. From these results, we show the following observations. First, these experiments show obviously that the ELRAP4TC algorithm is overall the fastest methods that offers satisfactory results. Second, ELRAP4TC outperforms HoMP in terms of their RMSE in Table \ref{Table 2}. Third, ADMM-t-SVD demands the highest cost in this experiment to the best RMSE result. These not only demonstrates the superiority of our ELRAP4TC, but also validate our recovery guarantee in Theorem \ref{Theorem 3.1} on video data.

\begin{table}[htb]
	\caption{\emph{Runing time and RMSE of tensor completion result on the basketball video}}\label{Table 2}
	\begin{center} 
		\begin{tabular}{|l|c|c|}
			\hline
			Completion Approach    & Running Time & RMSE   \\ \hline
			ADMM-t-SVD  & 124.76 & \textbf{11.7117}   \\ \hline  
			HoMP             & 38.40 & 29.6814  \\ \hline  
			ELRAP4TC (s=1)   & 17.55 & 20.2427   \\ \hline
			ELRAP4TC (s=3)   & \textbf{9.61} &  20.8546   \\  \hline
		\end{tabular}
	\end{center}
\end{table}

In the second case, we will explore the performance of our algorithm with a variety of specified missing ratio from $30\%$ to $90\%$. For each missing ratio, we test all algorithms 50 times and get the mean of their running time as vertical axis in Figure \ref{Figure 6} or their RMSE as vertical axis in Figure \ref{Figure 7}. The experiment setup is the same with the first experiment. From Figure \ref{Figure 6} and Figure \ref{Figure 7}, we can observe that the performance of all algorithms, for each specified missing ratio, are the almost same with the first experiment: ELRAP4TC is the fast method and outperforms HoMP to obtain the better accuracy solution.

\begin{figure}[h]
	\begin{center}
		\includegraphics[scale=0.45]{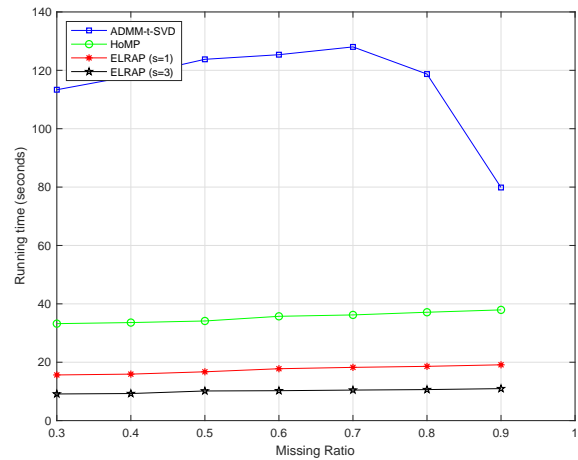}
		\caption{The trends of running times with different missing ratio.}
		\label{Figure 6}
	\end{center}
\end{figure}

\begin{figure}[h]
	\begin{center}
		\includegraphics[scale=0.45]{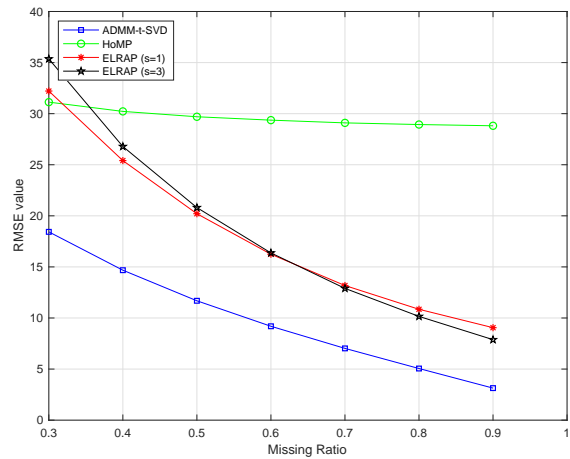}
		\caption{The trends of the RMSE values with different missing ratio.}
		\label{Figure 7}
	\end{center}
\end{figure}

In the third experiment, we give recovery performance comparison of all algorithm with a variety of frontal slice (video frame) number from $2$ to $40$. The experiment setup is the same with the first experiment. We test all algorithms 50 times, for each frontal slice number, and obtain the mean of their running time as vertical axis in Figure \ref{Figure 10} or their RMSE as vertical axis in Figure \ref{Figure 11}. Figures \ref{Figure 10} and \ref{Figure 11} show that the performances of all algorithms and each specified frontal slice generally are nearly the same as for our first experiment: namely ELRAP4TC is the fastest overall and it outperforms HoMP with better solutions for frontal slice numbers  above 8.

\begin{figure}[h]
	\begin{center}
		\includegraphics[scale=0.45]{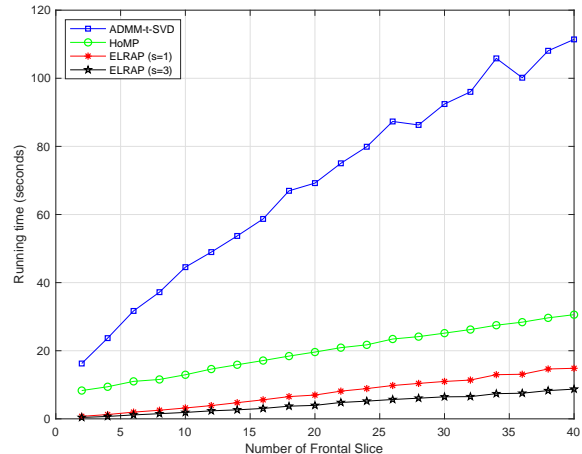}
		\caption{The trends of running times following different number of frontal slice.}
		\label{Figure 10}
	\end{center}
\end{figure}

\begin{figure}[h]
	\begin{center}
		\includegraphics[scale=0.45]{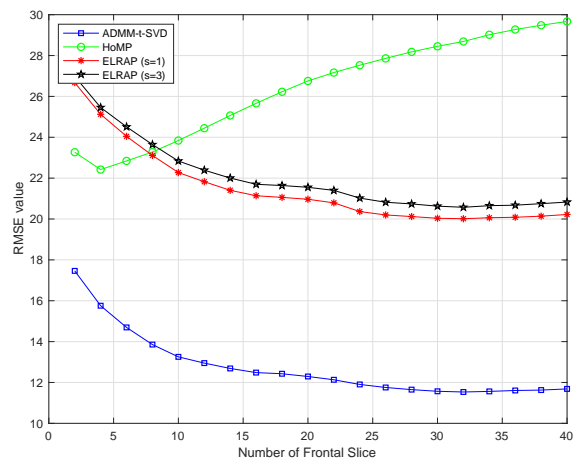}
		\caption{The trends of the RMSE values following different number of frontal slice.}
		\label{Figure 11}
	\end{center}
\end{figure}

\subsection{ Color Image Recovery}


In this subsection, we consider tensor completion based on the color image to test the stability of our algorithms. We format a $n_1 \times n_2$ sized color image as a tensor of size $n_1 \times n_2 \times 3$. Here the Matlab function, ${\rm imnoise}$, is used to generate blurring noise to the image and also add Gaussian noise with a mean zero and a standard deviation $\sigma=5e-3$. We will present that the recovery proformance of ELRAPTC is still satisfactory.

80 color images are used for the test from the Berkeley  Segmentation Dataset \cite{MFTM2001}. The sizes of images are $481 \times 321$ or $321 \times 481$. For each image, we remain $50 \%$ of the given pixels as known entries in $\Omega$ and set the desired minimal rank $r=100$. See Figure \ref{Figure 9} (b) for some sample images with noises. We compare our ELRAP4TC with ADMM-t-SVD and HoMP.

Figure \ref{Figure 8} gives the comparison of running time and RMSE on all 80 images. Some examples with the recovered images are represented in Figure \ref{Figure 9}. Based on these results, we have the observation that our ELRAP4TC is overall the fast methods to obtain reasonable solution in the presence of noise.


\begin{figure*}[ht]
	\begin{center}
		\hspace*{-25mm}		\includegraphics[scale=0.50]{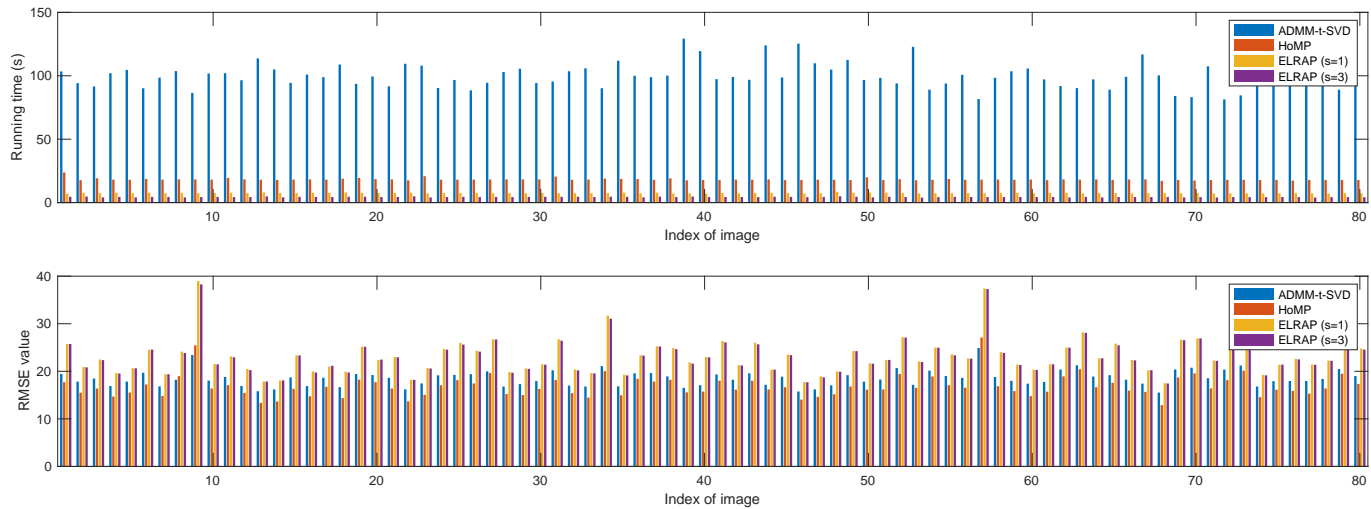}
		\caption{Comparison of running time (top) and the RMSE values (bottom) obtained by ADMM-t-SVD, HoMP, ELRAP4TC (s=1) and 	ELRAP4TC (s=3).}
		\label{Figure 8}
	\end{center}
\end{figure*}

\begin{figure*}[ht]
	\begin{center}
		\hspace*{-25mm}	\includegraphics[scale=0.78]{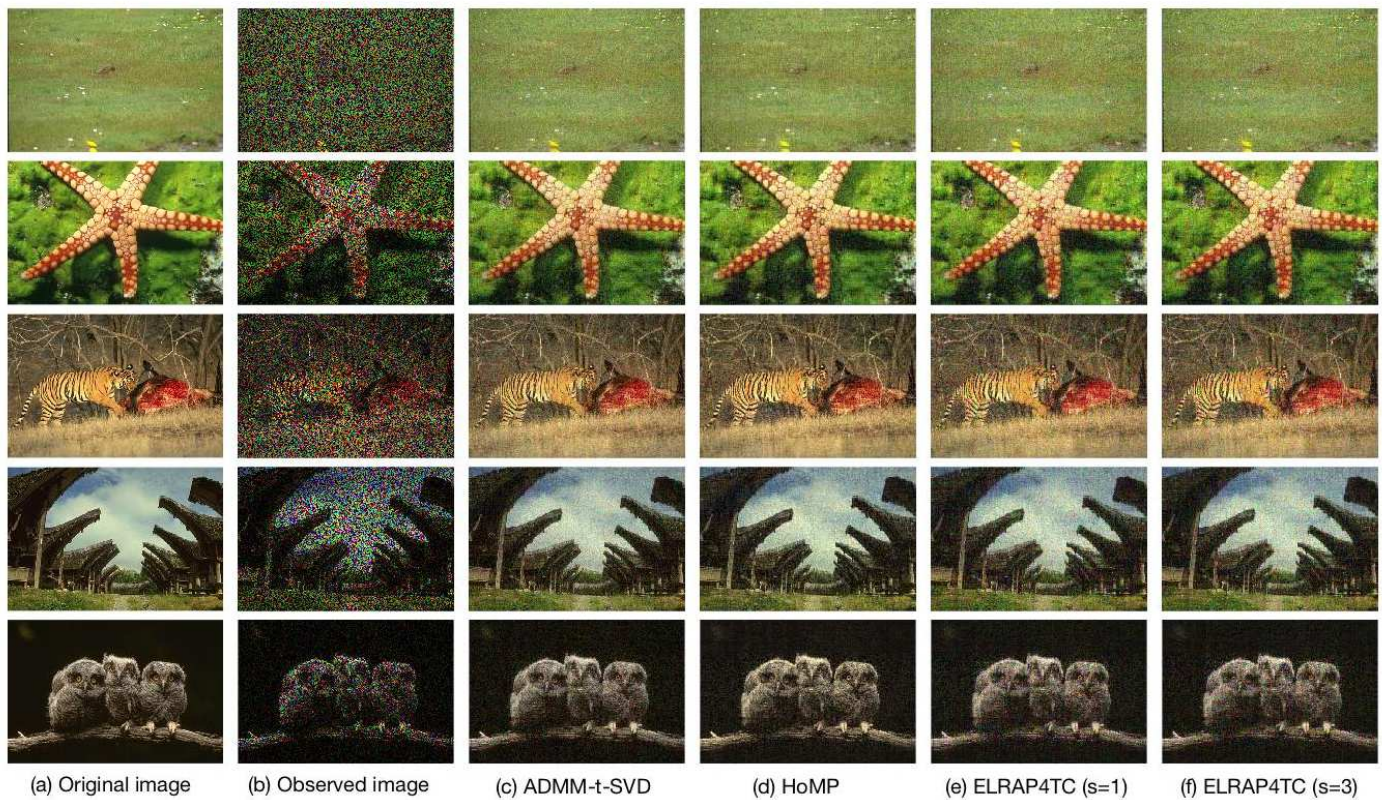}
		\caption{Recovery performance comparison on 5 example image. (1) Original image; (b) observed image; (c)-(f) recovered images by ADMM-t-SVD, HoMP, ELRAP4TC (s=1) and 	ELRAP4TC (s=3).}
		  \label{Figure 9}
	\end{center}
\end{figure*}

\begin{table}[htb]
	\caption{\emph{Comparison of running time on the 5 images in Figure \ref{Figure 9}}}\label{Table 3}
	\begin{center} 
	  \begin{tabular}{|c|c|c|c|c|}
			\hline
			Index    & ADMM-t-SVD & 	HoMP  &   \tabincell{c}{ELRAP4TC \\ (s=1)}  &  \tabincell{c}{ELRAP4TC \\ (s=3)}      \\ \hline
	    	1	 & 109.80&  22.25&     7.76&    \textbf{4.89}  \\ \hline  
		    2   & 94.21&    17.93&     7.24&     \textbf{4.50} \\ \hline  
			3   & 96.22&    18.09&     7.31&     \textbf{4.03}   \\ \hline
			4  & 97.36&   18.28&     7.23&     \textbf{4.35}  \\  \hline
			5  & 129.56&    17.93&     7.17&     \textbf{4.41}  \\  \hline
		\end{tabular}
	\end{center}
\end{table}

\begin{table}[htb]
	\caption{\emph{Comparison of RMSE on the 5 images in Figure \ref{Figure 9}}}\label{Table 3}
	\begin{center} 
		 \footnotesize{
		\begin{tabular}{|c|c|c|c|c|}
		\hline
	Index    & ADMM-t-SVD & 	HoMP  &   \tabincell{c}{ELRAP4TC \\ (s=1)} &  \tabincell{c}{ELRAP4TC \\ (s=3)}      \\ \hline
	1	 &  16.7631 &  \textbf{14.9213} &  19.1832  & 19.0462  \\ \hline  
	2   & 19.5448  & \textbf{18.2029}  & 23.2897  & 23.4329 \\ \hline  
	3   & 19.6845  & \textbf{17.7154}  & 25.3262 &  25.2199   \\ \hline
	4  & 18.9397 &  \textbf{18.0097} &  24.8622  & 24.4790 \\  \hline
	5  & 16.3832 &  \textbf{15.5059}  & 21.6744 &  21.3376  \\  \hline
		\end{tabular}
		  }
	\end{center}
\end{table}

\section{conclusion}
In this context, an efficient and scalable algorithms are proposed for  tensor completion and tensor sensing. In order to obtain the convergence of them, we define a new TRIP condition which is based on t-SVD. We show that subgaussian measurement ensemble satisfy the TRIP condition with high probability under the optimal bound on the number of measurements. Using this result, we present that both algorithms perform linear convergence rate. Numerical experiments on real datas are contained that show the accuracy and efficiency of our algorithms.

\bibliographystyle{ieeetr}
\bibliography{ref}

\begin{thebibliography}{10}

\bibitem{KB2009}
T.~G. Kolda and B.~W. Bader, ``Tensor decompositions and applications,'' {\em
  SIAM review}, vol.~52, no.~3, pp.~455--500, 2009.

\bibitem{H2012}
W.~Hackbusch, {\em Tensor Spaces and Numerical Tensor Calculus}, vol.~42.
\newblock 2012.

\bibitem{LMWY2013}
J.~Liu, P.~Musialski, P.~Wonka, and J.~Ye, ``Tensor completion for estimating
  missing values in visual data,'' {\em IEEE transactions on pattern analysis
  and machine intelligence}, vol.~35, no.~1, pp.~208--220, 2013.

\bibitem{GRY2011}
S.~Gandy, B.~Recht, and I.~Yamada, ``Tensor completion and low-n-rank tensor
  recovery via convex optimization,'' {\em Inverse Problems}, vol.~27, no.~2,
  p.~025010, 2011.

\bibitem{SDLS2013}
M.~Signoretto, Q.~T. Dinh, L.~D. Lathauwer, and J.~A.~K. Suykens, ``Learning
  with tensors: a framework based on convex optimization and spectral
  regularization,'' {\em Machine Learning}, vol.~94, no.~3, pp.~303--351, 2013.

\bibitem{TKK2010}
R.~Tomioka, H.~Kohei, and H.~Kashima, ``On the extension of trace norm to
  tensors, in: Nips workshop on tensor,'' {\em in: NIPS Workshop on Tensor,
  Kernels, and Machine Learning}, 2010.

\bibitem{PC2013}
A.~H. Phan, A.~Cichocki, P.~Tichavsk\'y, G.~Luta, and G.~Brockmeier, ``Tensor
  completion through multiple kronecker product decomposition,'' {\em in:
  ICASSP}, pp.~3233--3237, 2013.

\bibitem{YZ2014}
C.~H.~Z. M.~Yuan, ``On tensor completion via nuclear norm minimization,'' {\em
  arXiv preprint arXiv:1405.1773}, 2014.

\bibitem{ZEAHK2014}
Z.~Zhang, G.~Ely, S.~Aeron, N.~Hao, and M.~Kilmer, ``Novel methods for
  multilinear data completion and de-noising based on tensor-svd,'' {\em in
  Proc. IEEE Conf. Comput. Vis. Pattern Recognit.}, p.~3842–3849, 2014.

\bibitem{OTJ2010}
G.~Obozinski, B.~Taskar, and M.~I. Jordan, ``Joint covariate selection and
  joint subspace selection for multiple classification problems,'' {\em Stat.
  Comput.,}, vol.~20, no.~2, p.~231–252, 2010.

\bibitem{KS2008}
T.~Kolda and J.~Sun, ``Scalable tensor decompositions for multi-aspect data
  mining,'' {\em in Proc. 8th IEEE Int. Conf. Data Mining}, p.~363–372, 2008.

\bibitem{SPLCLQ2009}
J.~Sun, S.~Papadimitriou, C.~Lin, N.~Cao, S.~Liu, and W.~Qian, ``Multi- vis:
  Content-based social network exploration through multi-way visual analysis,''
  {\em in Proc. SIAM Int. Conf. Data Mining}, p.~1064–1075, 2009.

\bibitem{LLR1995}
N.~Linial, E.~London, and Y.~Rabinovich, ``The geometry of graphs and some of
  its algorithmic applications,'' {\em Combinatorica}, vol.~15, no.~2,
  p.~215–245, 1995.

\bibitem{ZSKA2018}
J.~Zhang, A.~K. Saibaba, M.~E. Kilmer, and S.~Aeron, ``A randomized tensor
  singular value decomposition based on the t-product, numerical linear algebra
  with applications,'' {\em IEEE Trans. Rel.}, vol.~25, no.~5, pp.~1--24, 2018.

\bibitem{LFCLLY2019}
C.~Lu, J.~Feng, Y.~Chen, W.~Liu, Z.~Lin, and S.~Yan, ``Tensor robust principal
  component analysis with a new tensor nuclear norm,'' {\em IEEE Transactions
  on Pattern Analysis and Machine Intelligence, DOI:
  10.1109/TPAMI.2019.2891760}, 2019.

\bibitem{ZA2017}
Z.~Zhang and S.~Aeron, ``Exact tensor completion using t-svd,'' {\em IEEE
  Transcations on Signal Processing}, vol.~65, no.~6, pp.~1511--1526, 2017.

\bibitem{D2016}
S.~Dirksen, ``Dimensionality reduction with subgaussian matrices: a unified
  theory,'' {\em Foundations of Computational Mathematics}, vol.~16, no.~5,
  pp.~1367--1396, 2016.

\bibitem{RSS2017}
H.~Rauhut, R.~Schneider, and Z.~Stojanac, ``Low rank tensor recovery via
  iterative hard thresholding,'' {\em Linear Algebra and its Application},
  vol.~523, pp.~220--262, 2017.

\bibitem{WLLFDY2015}
Z.~Wang, M.~J. Lai, Z.~Lu, W.~Fan, H.~Davulcu, and J.~Ye, ``Orthogonal rank-one
  matrix pursuit for low rank matrix completion,'' {\em SIAM Journal on
  Scientific Computing}, vol.~37, no.~1, pp.~A488--A514, 2015.

\bibitem{XX2017}
A.-B. Xu and D.~Xie, ``Low-rank approximation pursuit for matrix completion,''
  {\em Mechanical Systems and Signal Processing}, vol.~95, pp.~77--89, 2017.

\bibitem{KM2011}
M.~K. Kilmer and M.~C. D., ``Factorization strategies for third-order tensor,''
  {\em Linear Algebra and its Application}, vol.~435, no.~3, pp.~641--658,
  2011.

\bibitem{ZWHWW2019}
F.~Zhang, W.~Wang, J.~Huang, Y.~Wang, and J.~Wang, ``Rip-based performance
  guarantee for low-tubal-rank tensor recovery,'' {\em arXiv preprint
  arXiv:1906.01774}, 2019.

\bibitem{V2012}
R.~Vershynin, Y.~Eldar, and G.~Kutyniok, {\em Compressed sensing: Theory and
  applications" in Introduction to the Non-Asymptotic Analysis of Random
  Matrices}.
\newblock 2012.

\bibitem{CP2011}
E.~J. Cand\`es and Y.~Plan, ``Tight oracle inequalities for low-rank matrix
  recovery form a minimal number of noisy random measurements,'' {\em IEEE
  Transaction on Information Theory}, vol.~57, no.~4, pp.~2342--2359, 2011.

\bibitem{YMS2015}
Y.~Yang, S.~Mehrkanoon, and J.~A.~K. Suykens, ``Higher order matching pursuit
  for low rank tensor learning,'' {\em arXiv preprint arXiv:1503.02216}, 2015.

\bibitem{MFTM2001}
D.~Martin, C.~Fowlkes, D.~Tal, and J.~Malik, ``A database of human segmented
  natural images and its application to evaluating segmentation algorithms and
  measuring ecological statistics.,'' {\em In Proc. IEEE Int'l conf. Computer
  Vision}, vol.~2, pp.~416--423, 2001.

\end{thebibliography}



\end{document}